%% file: surgery1.tex
\documentclass[10pt]{amsart}
\input{header}

\title{A surgery formula for Seiberg-Witten invariants}
\author{Haochen Qiu}

\begin{document}

\maketitle

\begin{abstract}
    We prove a surgery formula for the ordinary Seiberg-Witten invariants of smooth $4$-manifolds with $b_1 =1$. Our formula expresses the Seiberg-Witten invariants of the manifold after the surgery, in terms of the original Seiberg-Witten moduli space cut down by a cohomology class in the configuration space. This formula can be used to find exotic smooth structures on nonsimply connected $4$-manifolds, and gives a lower bound of the genus of an embedding surface in nonsimply connected $4$-manifolds. In forthcoming work, we will extend these results to give a surgery formula for the families Seiberg-Witten invariants.
\end{abstract}

\tableofcontents
\section{Introduction}

Let $\gamma$ be a loop in a closed smooth $4$-manifold $X$. A surgery along $\gamma$ is removing a neighborhood of $\gamma$ with a framing (trivialization of the normal bundle), and gluing back a copy of $D^2\times S^2$. For example, a surgery along $ S^1\times\{pt\} \subset S^1\times S^3 $ would produce $S^4$, while a surgery along a trivial loop on $S^4$ may produce $S^2\times S^2$ or $\CP^2 \# \overline{\CP^2}$.  Such surgery establishes relations between lots of $4$-manifolds. This work describes how a surgery can preserve exotic phenomena.

The tool we use comes from the Seiberg-Witten equations introduced by Seiberg and Witten (\cite{Seiberg1994ElectricM}). The input of the equation for $X$ includes a $\text{Spin}^c$-structure (they are related to elements in $H^2(X;\Z)$), a $U(1)$-connection, and a ``spinor''. The set of equivalence classes of $U(1)$-connections and spinors under the ``gauge group'' $\text{Map}(X,S^1)$ is called the \textbf{configuration space} (denoted by $\mathcal{B}$), which is a fiber bundle with fiber $\CP^\infty$ and base a torus $T^{b_1(X)}$. The Seiberg-Witten equations depend on a tuple consisting of a metric and a perturbing $2$-form, which is called a \textbf{parameter}. The solution of this equation with a suitable parameter is a smooth compact manifold in the configuration space. This manifold is called the Seiberg-Witten \textbf{moduli space} (denoted by $\M$). Its dimension is computed by the Atiyah-Singer index theorem, and if it is even, we can integrate a poduct of $c_1(\CP^\infty)$ on the moduli space and get the so-called \textbf{Seiberg-Witten invariant} (when the dimension is $0$, the integral just counts the points with signs). This is an invariant under diffeomorphism. Many examples of exotic $4$-manifolds were found by computing this invariant for two homeomorphic manifolds.

In this paper, we generalize $SW$ to $1$-dimensional moduli space, such that the new invariant (we call it the cut-down invariant $SW^\Theta$) can detect exotic phenomena. Then we prove a surgery formula that shows how a surgery relates $SW^\Theta$ to $SW$.

\subsection{Main result}
For a $4$-manifold $X$ with $H^1(X;\Z)=\Z$, suppose $\ss$ is a $\text{Spin}^c$-structure such that $\dim \M(X,\ss)=1$. That is,
\[\label{equ:dimAssumption}
\tag{\textbf{Dim}}
\frac{1}{4}\left(\int_{X} c_1(\L)^2 - 2(\chi(X)+3\sigma(X))\right) =1,
\]
where $\L$ is the determinant line bundle of $\ss$, $\chi(X)$ is the Euler characteristic of $X$ and $\sigma(X)$ is the signature of $X$. The configuration space is homotopy equivalent to a bundle over $S^1$ with fiber $ \CP^\infty$. Let $\Theta$ be the pullback of a generator of $H^1(S^1;\Z)$. Define the cut-down Seiberg-Witten invariant $SW^\Theta(X,\ss)$ be the integral of $\Theta$ on $\M(X,\ss) $, that is,
\[
SW^\Theta(X,\ss) :=  \langle\Theta,[ \M(X,\ss) ] \rangle.
\]
We prove that this invariant detects exotic smooth structures.

Let $\gamma \subset X$ be a loop that represents a generator of $H_1(X;\Z)/\text{torsion} = \Z$. Suppose a surgery along $\gamma$ produces $X'$. We show that there is a unique $\text{spin}^c$ structure $\ss'$ of $X'$ that coincides with $\ss$ on their common part (see Section \ref{subsection:Spin}). Under the dimension assumption (\ref{equ:dimAssumption}), The SW moduli space $\M(X, \ss)$ of $X$ would be $1$-dimensional. Since the surgery kills the first cohomology group, $H^1(X';\Z)=0$ and therefore $\dim \M(X',\ss')=0$. Hence $SW(X',\ss')$ is defined by counting points in $ \M(X',\ss)$. The main theorem of this project is 
\begin{theorem}
$SW^\Theta(X,\ss)=SW(X',\ss')$.
\end{theorem}
Geometrically, this means that cutting a $1$-handle of a manifold, corresponds to cutting its Seiberg-Witten moduli space.

This is proved by applying the classical gluing result in Nicolaescu's book \cite{Nicolaescu2000NotesOS} twice. Let $S^1 \times D^3$ be a neighborhood of $\gamma$, and let $X_0 = X - S^1 \times D^3$. Then gluing $X_0$ with $S^1 \times D^3$ produces $X$, while gluing $X_0$ with $D^2 \times S^2$ produces $X'$. The classical gluing result says, if a certain ``obstruction space'' is trivial on $X_0$, then $\M(X)$ is the fiber product $\M(X_0)\times_{\M(S^1 \times S^2)} \M(S^1 \times D^3)$ while $\M(X')$ is the fiber product $\M(X_0)\times_{\M(S^1 \times S^2)} \M(D^2 \times S^2)$. We prove that since $\gamma$ is homologically nontrivial, for generic parameters such obstruction space is $0$-dimensional. Furthermore, we can choose suitable metrics such that $\M(S^1 \times D^3) \to \M(S^1 \times S^2)$ is the identity map of a circle, and $\M(D^2 \times S^2) \to \M(S^1 \times S^2)$ is the inclusion of one point into a circle. Hence if we cut $\M(X')$, we get $\M(X)$, and the theorem follows.

\subsection{Applications}
As lots of exotic smooth structures are detected by $SW$, we can now generalize those results to nonsimply connected manifolds, for example:
\begin{corollary}
$E(n) \# (S^1 \times S^3)$ admits infinitely many exotic smooth structures.
\end{corollary}

The surgery formula has another kind of application. The generalized adjunction formula gives a lower bound of the genus of a surface in a smooth $4$-manifold $X$, by Seiberg-Witten basic classes. A characteristic element $K\in H^2(X;\Z)$ is a Seiberg-Witten basic class of $X$ if $SW(X,K)\neq 0$. We can generalize this concept and the adjunction formula to odd dimensional moduli space. 
\begin{corollary}
Suppose that $\Sigma$ is an embedded, oriented, connected, homologically nontrivial surface in $X$ with genus $g(\Sigma)$ and self-intersection $[\Sigma]^2 \ge 0$. Suppose that $H^1(X;\Z) \cong \Z$, and that $K\in H^2(X;\Z)$ is a characteristic element such that:
\begin{enumerate}
\item[(1)] $SW(X,K)\neq 0$ for $\dim \M(X,K) =2n$,
\item[(2)]  or $SW^\Theta(X,K)\neq 0$ for $\dim \M(X,K) =2n+1$.
\end{enumerate}
Then we have
\[
2g(\Sigma) -2 \ge [\Sigma]^2 + |K([\Sigma])|.
\]
\end{corollary}

This gives a sharper lower bound of the genus of an embedding surface in nonsimply connected $4$-manifolds than the adjunction formula, which has only case (1).

\subsection*{Acknowledgements}
The author wants to express gratitude to his advisor Daniel Ruberman for his suggestions and support throughout the preparation of this work.  The author also wants to thank Hokuto Konno for his suggestion regarding the references. This work is partially supported by NSF grant DMS-1952790.

\section{Setup for the $1$-surgery formula}
\subsection{$\text{Spin}^\C$ structure }\label{subsection:Spin}
The definition of the Seiberg-Witten moduli space depends on a choice of the $\text{Spin}^\C$ structure, so we first review the theory of the $\text{Spin}^\C$ structure. 
Definitions in this subsection can be found in section 1.4.2 and 2.4.1 of \cite{gompf19994}. We also provide some auxilary examples (Example \ref{example:spin} and Remark \ref{remark:spincline}). The main theorem in this subsection is Theorem \ref{thm:changeOfSpinc}. It deals with the change of $\text{Spin}^\C$ structures by a $1$-surgery.

To understand the $\text{Spin}^\C$ structure, we first review the theory of the $\text{spin}$ structure. 
\begin{definition}
    \[
        \text{Spin}(4) =  SU(2) \times SU(2) 
    \]
    is called the $\text{spin}$ group of dimension $4$.
\end{definition}
Note that, $\text{Spin}(4)$ is the connective double cover of $SO(4) = SU(2) \times SU(2) / \{\pm (I,I)\}$. 

\begin{remark}\label{remark:spin}
    Double covers of $X$ correspond to $H^1(X;\Z_2) =[X, \RP^\infty]$. The correspondence is given by the sphere bundle of pull back of the universal line 
    bundle (tautological line bundle over $\RP^\infty$). For $SO(4)$, $H^1(SO(4);\Z_2) =[SO(4), \RP^\infty] = \Z_2$. So the double covers of $SO(4)$ are charecterized by the homotopy class of the image of the nontrivial loop of $SO(4)$ in $\RP^\infty$. 
    If that loop is homotopic to a constant loop in $\RP^\infty$, then the corresponding double cover is $SO(4)\sqcup SO(4)$. 
    If that loop is homotopic to the $1$-cell of $\RP^\infty$, then the double cover is $\text{Spin}(4)$.
\end{remark}

\begin{definition}
    A $\text{spin}$ structure $\ss$ on a $4$-manifold $M$ is a principal $\text{Spin}(4)$-bundle 
    $P_{\text{Spin}(4)} \to M$, with a bundle map from $P_{\text{Spin}(4)}$ to the frame bundle $P_{SO(4)}$ of $M$, which restricts
    to the double cover $\rho: \text{Spin}(4) \to SO(4)$ on each fiber.
    \end{definition}

Note that $P_{\text{Spin}(4)}$ is a double cover of $P_{SO(4)}$, which restricts to the double cover $\rho: \text{Spin}(4) \to SO(4)$ on each fiber. By Remark \ref{remark:spin}, 
this corresponds to an element in $H^1(P_{SO(4)};\Z_2)=[P_{SO(4)}, \RP^\infty]$ which restricts to the nontrivial element in $H^1(SO(4);\Z_2)=[SO(4), \RP^\infty]$ on each fiber. 
From the Leray-Serre spectral sequence, we have the following exact sequence:
\[
    0\to H^1(M, \Z_2) \to H^1(P_{SO(4)}, \Z_2) \stackrel{i^*}{\to} H^1(SO(4), \Z_2) \stackrel{\delta}{\to} H^2(M, \Z_2).
\]
Here $\delta(1) = w_2(P_{SO(4)})$, and $i^*$ is the restriction map. By the discussion above, the set of $\text{spin}$ structures on $M$ 
is in one-to-one correspondence with $(i^*)^{-1}(1)$. When $\delta(1) = w_2(P_{SO(4)}) = 0$, $(i^*)^{-1}(1)$ is nonempty, 
and 
\begin{align*}
\#(i^*)^{-1}(1) &=\#(i^*)^{-1}(0)\\
& = \# \im (H^1(M, \Z_2) \to H^1(P_{SO(4)}, \Z_2)) \\
&= \# H^1(M, \Z_2) .
\end{align*}
So the set of $\text{spin}$ structures on $M$ 
is in noncanonical one-to-one correspondence with $H^1(M, \Z_2)$. 
When $\delta(1) = w_2(P_{SO(4)}) \neq 0$, $(i^*)^{-1}(1)$ is empty.

\begin{example}\label{example:spin}
    Let $M = \S^1 \times \R^3$. Then $w_2(TM) = 0$ and $H^1(M, \Z_2) = \Z_2$. Hence there are two $\text{spin}$ structures on $M$. 
    They are principal $\text{Spin}(4)$-bundles that cover the trivial bundle $P_{SO(4)} = M \times SO(4)$, and the covering maps are nontrivial on each fiber. 
    Namely, the preimage of the nontrivial loop of $SO(4)$ is $\S^1$, and the covering maps restrict to this preimage are both 
    \begin{align*}
    \S^1 &\stackrel{2}{\to} \S^1\\
     z &\mapsto z^2.
     \end{align*}
    These two $\text{spin}$ structures are distinguished by the covering maps on the $\S^1$ factor of $M$. They are nontrivial double cover 
    $\S^1 \stackrel{2}{\to} \S^1 \subset SO(4)$ and trivial double cover $\S^1\sqcup \S^1 {\to} \S^1 \subset SO(4)$, respectively. 
    
    We can construct these principal $\text{Spin}(4)$-bundles explicitly. Let $\{U_\alpha, U_\beta\}$ 
    be a good cover of $M$ such that $U_\alpha$ and $U_\beta$ are diffeomorphic to $\R \times \R^3$. Let $U_0 \sqcup U_1 = U_\alpha \cap U_\beta$. 
    Let $P_{SO(4)}$ be the frame bundle of $M$ with local trivialization on $\{U_\alpha, U_\beta\}$ and transition functions $g_i: U_i \to SO(4)$ for $i = 0,1$. 
    Fix $m_i \in U_i$ and a lift 
    \[\widetilde{g_i(m_i)} \in \text{Spin}(4)\] 
    for ${g_i}(m_i)$ respectively. Since $P_{\text{Spin}(4)} \to P_{SO(4)}$ is a fibration and $U_i$ is contractible, 
    we can lift $g_i$ to a map $\tilde{g_i}: U_i \to \text{Spin}(4)$ such that \[\tilde{g_i}(m_i) = \widetilde{g_i(m_i)}.\] 
    This gives the transition functions for a principal $\text{Spin}(4)$-bundle $P_{\text{Spin}(4)}$ over $M$ which is locally trivial on $\{U_\alpha, U_\beta\}$. 
    To construct another principal $\text{Spin}(4)$-bundle, we choose the same lift of $g_0(m_0)$ but a different lift of $g_1(m_1)$.

    For example, if 
    \[g_i(m) = [I,I] \in SO(4) = SU(2) \times SU(2) / \{\pm (I,I)\}\] 
    and 
    \[\tilde{g_i}(m) = (I,I)\in \text{Spin}(4) =  SU(2) \times SU(2)\]
    for any $i$ and $m\in U_i$, then the principal $\text{Spin}(4)$-bundle is trivial. For the loop $l = \S^1 \times \{0\} \times \{I\} \subset M \times  SO(4) =  P_{SO(4)}$, 
    the preimage of $l$ under the double cover $P_{\text{Spin}(4)} \to P_{SO(4)}$ is $\S^1\sqcup \S^1 \subset P_{\text{Spin}(4)}$. On the other hand, if
    \begin{align*}
        \tilde{g_0}(m) &= (I,I), m\in U_0\\
        \tilde{g_1}(m) &= (-I,-I), m\in U_1,
    \end{align*}
    then geometrically, when a particle runs along $l$, it's preimage under the double cover $P_{\text{Spin}(4)} \to P_{SO(4)}$ changes to another orbit when this particle passes $U_1$. Thus the preimage of $l$ is a single $\S^1 \subset P_{\text{Spin}(4)}$. 
    This example shows that the set of $\text{spin}$ structures on $M$ 
is in one-to-one correspondence with $H^1(M, \Z_2)$. Moreover, such correspondence is noncanonical: There is not a priori choice of the lift of $g_i(m_i)$.

\end{example}

Now we introduce the $\text{spin}^\C$ structure.

\begin{definition}
    \[
        \text{Spin}^\C(4) = \{(A,B)\in U(2) \times U(2) ; \text{det}(A) = \text{det}(B) \}
    \]
    is called the $\text{spin}^\C$ group of dimension $4$. 
\end{definition}

Note that, $\text{Spin}^\C(4)$ is isomorphic to $S^1 \times SU(2) \times SU(2) / \{\pm (1,I,I)\}$, while $SO(4)$ is isomorphic to 
$ SU(2) \times SU(2) / \{\pm (I,I)\}$. Hence we have an $S^1$-fiberation
\begin{align}\label{equ:rhoc}
    \rho^c:\text{Spin}^\C(4) &\to SO(4)\\
    [(z,A,B)] &\mapsto [(A,B)].
\end{align}

\begin{definition}\label{def:spinc}
A $\text{spin}^\C$ structure $\ss$ on a manifold $M$ is a principal $\text{Spin}^\C(4)$-bundle 
$P_{\text{Spin}^\C(4)} \to M$, with a bundle map from $P_{\text{Spin}^\C(4)}$ to the frame bundle $P_{SO(4)}$ of $M$, which restricts
to $\rho^c$ on each fiber.
\end{definition}

Looking at the definition of $\rho^c$, we find that a $\text{spin}^\C$ structure contains one more infomation than the frame bundle:
\begin{definition}
Let 
\begin{align}
    \det: \text{Spin}^\C(4) &\to S^1 \\
    \label{equ:determinant}[(z,A,B)] &\mapsto z^2.
\end{align}
The line bundle $\L = P_{\text{Spin}^\C(4)} \times_{\det} \C $ is called the determinant line bundle associated to the $\text{spin}^\C$ structure $\ss$. 
\end{definition}

A $\text{spin}^\C$ structure is actually a double cover of the frame bundle tensor the determinent line bundle. We have 
an exact sequence
\begin{align}\label{equ:doublespinc}
    1\to \Z_2 \to \text{Spin}^\C(4) &\stackrel{\rho'}{\to} S^1 \times SO(4) = SO(2) \times SO(4) \to 1\\
    [(z,A,B)] &\mapsto (z^2,[(A,B)]).
\end{align}
The double cover $\rho'$ can be extended to a double cover of $SO(6)$ (see page 56 of \cite{gompf19994}). Hence the $\text{spin}^\C$ structure exists if and only if
the second Stiefel-Whitney class $w_2(P_{\S^1\times SO(4)})$ vanishes, by the theory of the existence of spin structures metioned above. Namely, 
\begin{align}
    w_2(P_{\S^1\times SO(4)}) &= w_2(P_{\S^1}) + w_2(P_{SO(4)}) \\
    &= w_2(\L) + w_2(TM)\\
    &= 0 \in \Z/2.
\end{align}
Namely, $ w_2(TM) \equiv c_1(\L) \mod 2$. An integral cohomology class congruent to $w_2(TM)$ is called characteristic element. The set of characteristic elements is nonempty for any $4$-manifold (see Proposition 5.7.4 of \cite{gompf19994}). Thus the $\text{spin}^\C$ structure always exists.

\begin{remark}\label{remark:spincline}
Different choices of the double covers of $P_{\S^1\times SO(4)}$ (with the covering map $\rho'$ fiberwise) do not always give different $\text{spin}^\C$ structures. 
Indeed, the set of $\text{spin}^\C$ structures over $M$ is in (non-canonical) one-to-one correspondence with the isomorphism classes of complex line bundles over $M$. 
Recall that,
\begin{align}
    SO(4)  &= SU(2) \times SU(2) / \{\pm (I,I)\}\\
    \text{Spin}(4) &=  SU(2) \times SU(2) \\
    \text{Spin}^\C(4) &=  S^1 \times SU(2) \times SU(2) / \{\pm (1,I,I)\} = S^1 \times  \text{Spin}(4) / \{\pm (1,I)\}. \label{equ:spinc}
\end{align}
Thus the transition functions of a principal $\text{Spin}^\C(4)$-bundle over $M$ are given by $[z_{\alpha\beta},g_{\alpha\beta}]$ 
where $z_{\alpha\beta}: U_{\alpha\beta}\to S^1$ and $g_{\alpha\beta}: U_{\alpha\beta}\to \text{Spin}(4)$ for a good cover $\{U_{\alpha}\}$. 
Suppose we have two $\text{spin}^\C$ structures 
\begin{align*}
&P^{(1)}_{\text{Spin}^\C(4)} \to M\\ 
&P^{(2)}_{\text{Spin}^\C(4)} \to M
\end{align*}
with transition functions 
$[z_{\alpha\beta}^{(1)},g_{\alpha\beta}^{(1)}]$ and $[z_{\alpha\beta}^{(2)},g_{\alpha\beta}^{(2)}]$ respectively. 
Note that by the definition of the $\text{spin}^\C$ structure, 
\[
\rho^c([z_{\alpha\beta}^{(i)},g_{\alpha\beta}^{(i)}])=[g_{\alpha\beta}^{(i)}] \in SO(4)
\] 
would be the transition functions of the frame bundle $P_{SO(4)}$. 
Hence we have either 
\[g_{\alpha\beta}^{(1)} = g_{\alpha\beta}^{(2)}\in  \text{Spin}(4)\]
or 
\[g_{\alpha\beta}^{(1)} = -g_{\alpha\beta}^{(2)} \in  \text{Spin}(4).\] 
If it's the latter case, we can always choose a different representative of $[z_{\alpha\beta}^{(2)},g_{\alpha\beta}^{(2)}]$. Thus we can assume that 
$g_{\alpha\beta}^{(1)} = g_{\alpha\beta}^{(2)}$. Then 
\[
\theta_{\alpha\beta} =  z_{\alpha\beta}^{(2)}/z_{\alpha\beta}^{(1)}
\]
would give 
the transition functions of a complex line bundle $\L$ over $M$, such that 
\[
P^{(1)}_{\text{Spin}^\C(4)}\otimes \L \cong P^{(2)}_{\text{Spin}^\C(4)}.
\]
(This shows that the action of $H^2(M;\Z) = [M,\C P^\infty]$ on the set of $\text{spin}^c$-structures is transitive. Actually this action is also free.)

By the definition of the determinant line bundle, 
\[
\det (P^{(1)}_{\text{Spin}^\C(4)}\otimes \L) = \det(P^{(1)}_{\text{Spin}^\C(4)}) \otimes \L^2.
\]
Hence 
\[
c_1(\det (P^{(1)}_{\text{Spin}^\C(4)}\otimes \L)) = c_1( \det(P^{(1)}_{\text{Spin}^\C(4)})) + 2c_1(\L).
\]
When $H^2(M;\Z)$ has no $2$-torsion, $2c_1(\L) = 0$ iff $c_1(\L) = 0$, iff $\L$ is trivial. Hence $c_1\circ \det$ is injective. If $P^{(1)}_{\text{Spin}^\C(4)}$ and $P^{(2)}_{\text{Spin}^\C(4)}$ are two different choices of the double covers of $P_{\S^1\times SO(4)}$ (with the covering map $\rho'$ fiberwise), then the difference line bundle $\L$ has transition functions $\theta_{\alpha\beta} = \pm 1$ such that $\L^2$ is trivial. Hence 
\[
\det(P^{(2)}_{\text{Spin}^\C(4)}) = \det (P^{(1)}_{\text{Spin}^\C(4)}\otimes \L) = \det(P^{(1)}_{\text{Spin}^\C(4)}) \otimes \L^2 =  \det (P^{(1)}_{\text{Spin}^\C(4)})
\]
and therefore $c_1\circ \det$ sends them to the same element. Hence they are isomorphic $\text{spin}^\C$ structures.

In conclusion, although it seems that by (\ref{equ:doublespinc}) and (\ref{equ:spinc}) a $\text{spin}^\C$ structure encodes some infomation of the spin structure, and by Example \ref{example:spin}, each element of $H^1$ would produce a different spin structure, but that difference comes from the different choice of the lift of $\text{Spin}(4) \to SO(4)$, which can be passed to the difference of the complex line bundle in $P_{\text{Spin}^\C(4)}$.
\end{remark}

For a $1$-surgery along a nontrivial loop, all $\text{spin}^\C$ structures can be extended to the new manifold uniquely.
\begin{theorem}\label{thm:changeOfSpinc}
Let $X$ be any $4$-manifold with $H^1(X;\Z) = \Z$. Let $\alpha$ be a generator of $H^1(X;\Z)$. Let $\gamma$ be the loop we choose to do the surgery, with $ \langle \alpha, \gamma \rangle = 1$. Let $N = \S^1 \times D^3$ be a tubular neighborhood of $\gamma$ with radius sufficiently small.
    Let $X_0$ be the complement of $N$. Let $X'= X_0 \cup_{\S^1 \times \S^2} (D^2 \times \S^2)
    $ be the manifold obtained by doing the surgery on $X$ along $\gamma$. Let $\ss$ be any $\text{Spin}^\C$ structure over $X$ and $\L$ be the corresponding determinant line bundle. Let $\mathcal{S}(X')$ be the set of $\text{spin}^\C$ structures on $X'$, and define
    \[
    \mathcal{S}(X', \ss) := \{\Gamma \in \mathcal{S}(X') ; \left.\Gamma\right|_{X_0} =\left. \ss\right|_{X_0}\}.
    \]
    Then $ \mathcal{S}(X', \ss) $ contains a unique (up to an isomorphism) $\text{Spin}^\C$ structure $\ss'$ over $X'$,
    and the determinant line bundle $\L'$ associated to $\ss'$ satisfies
    \[
    \langle c_1(\L')^2,X' \rangle = \langle c_1(\L)^2 ,X \rangle.
    \]
In particular, above results do not depend on the framing of the $1$-surgery.
\end{theorem}

\begin{proof}
We first show that $\mathcal{S}(X', \ss)$ is nonempty. Let $\ss'$ be any $\text{Spin}^\C$ structure over $X'$. By Remark \ref{remark:spincline}, 
the difference between $\ss'|_{X_0}$ and $\ss|_{X_0}$ is a complex line bundle $L_0$ over $X_0$, namely, $\ss'|_{X_0} \otimes L_0 = \ss|_{X_0}$. 

We claim that $L_0$ can be extended to a complex line bundle $L'$ over $X'$. Indeed, for the inclusions 
\begin{align*}
    i_\partial: \partial X_0 = \S^1 \times \S^2 &\to D^2 \times \S^2\\
    i: X_0 &\to X',
\end{align*} 
the induced homomorphisms
\begin{align*}
    i_\partial^*: H^2(D^2 \times \S^2) &\to H^2(\S^1 \times \S^2)\\  
    i^*: H^2(X') &\to H^2(X_0)
\end{align*}
are all isomorphisms. This follows from the following Mayer-Vietoris sequence (the last three terms form a split short exact sequence, because the dual of $\alpha$ in $X$ implies that $\{pt\} \times \S^2 \subset \S^1\times \S^2 = \partial X_0$ bounds a $3$-manifold in $X_0$, and therefore the inclusion $\S^1 \times \S^2 \to X_0$ induces a trivial map of $H^2$):
\begin{center}
\begin{tikzpicture}[commutative diagrams/every diagram]
\node (P0) at (0cm, 0cm) {$H^1(X_0)\oplus H^1(D^2 \times \S^2)$};
\node (P1) at (-1cm, -0.1cm) {};
\node (P2) at (3.3cm, 0cm) {$H^1(\S^1 \times \S^2)$} ;
\node (P6) at (3.3cm, -1cm) {$\Z$} ;
\node (P3) at (5.5cm, 0cm) {$H^2(X')$};
\node (P4) at (-1cm, -1cm) {$\Z$};
\node (P5) at (8.5cm, 0cm) {$H^2(X_0)\oplus H^2(D^2 \times \S^2)$};
\node (P7) at (8.7cm, -0.1cm) {};
\node (P8) at (8.7cm, -1cm) {$\Z$};
\node (P9) at (12cm, 0cm) {$H^2(\S^1 \times \S^2)$};
\node (P10) at (12cm, -1cm) {$\Z$};
\path[commutative diagrams/.cd, every arrow, every label]
(P0) edge node {} (P2)
(P2) edge node {$0$} (P3)
(P1) edge node {$\cong$} (P4)
(P3) edge node {} (P5)
(P2) edge node {$\cong$} (P6)
(P4) edge node {$\cong$} (P6)
(P7) edge node {$\cong$} (P8)
(P5) edge node {} (P9)
(P8) edge node {$ i_\partial^*$} (P10)
(P9) edge node {$\cong$} (P10);
\end{tikzpicture}
\end{center}
Topologically, the dual of $c_1(L_0)|_{\partial X_0}$ is some copies of $\S^1 \times \{pt\} \subset \S^1 \times \S^2 = \partial X_0$, and they can be extended to $D^2 \times \{pt\} \subset D^2 \times \S^2$. In conclusion, there exists a cohomology class in $H^2(X',\Z)= [X' , \CP^\infty]$ that restricts to $c_1(L_0) \in H^2(X_0,\Z)$, and by the property of the universal complex line bundle over $\CP^\infty$, the pullback $L'$ is a complex line bundle over $X'$ that restricts to $L_0$. Therefore, we have
\[
\left.(\ss'\otimes L')\right|_{X_0}  =\left.\ss'\right|_{X_0} \otimes L_0 =  \left.\ss\right|_{X_0}.
\]
So $\ss'\otimes L' \in \mathcal{S}(X', \ss)$.

Next, we prove that all elements in $\mathcal{S}(X', \ss)$ are isomorphic. Let $\ss'_{(1)}, \ss'_{(2)} \in \mathcal{S}(X', \ss)$. Let $L'$ be a complex line bundle on $X'$ such that 
\[
 \ss'_{(1)} \otimes L' = \ss'_{(2)}.
 \]
Then
\begin{align*}
\left.\ss'_{(2)}\right|_{X_0}  &=  \left.(\ss'_{(1)} \otimes L')\right|_{X_0}\\
 &= \left.\ss'_{(1)}\right|_{X_0} \otimes \left.L'\right|_{X_0} \\
 &= \left.\ss\right|_{X_0} \otimes \left.L'\right|_{X_0}\\
 &= \left.\ss'_{(2)}\right|_{X_0}\otimes \left.L'\right|_{X_0}.
\end{align*}
Remark \ref{remark:spincline} shows that the action of $H^2(X_0)= [X_0,\C P^\infty]$ on $\mathcal{S}(X_0)$ is free and transitive. Hence $c_1( L'|_{X_0})= 0\in H^2(X_0,\Z)$. Note that $i^*(c_1( L')) = c_1( L'|_{X_0})$ and $i^*$ is an isomorphism. Therefore $c_1( L')=0\in H^2(X',\Z)$. So $L'$ is trivial and $\ss'_{(1)}= \ss'_{(2)}$.

Lastly, we show that 
\[
    \langle c_1(\L')^2,X' \rangle = \langle c_1(\L)^2 ,X \rangle.
    \]
The intersection between a generic section of $\L$ and the zero section is a $2$-manifold $\Sigma \subset X$. For dimension reason we can assume $\gamma \cap \Sigma  = \emptyset$. By choosing a small enough neighborhood of $\gamma$ we can further assume $\Sigma \subset X_0$. $ \langle c_1(\L)^2 ,X \rangle$ is the self-intersection $[\Sigma]^2$ of $\Sigma$.

Since $\ss'|_{X_0} = \ss|_{X_0}$, $\L'|_{X_0} = \det(\ss')|_{X_0} =\det(\ss)|_{X_0} =\L|_{X_0}$. As a complex line bundle, $\L|_{\S^1\times D^3}$ must be trivial. Hence it's a trivial line bundle over $\partial X_0$. Since $ i^*: H^2(X') \to H^2(X_0)$ is an isomorphism, $\L'$ is the unique extension of $\L'|_{X_0}=\L|_{X_0}$, and therefore it must extend $\L|_{\partial X_0}$ trivially. Hence the generic section of $\L|_{X_0}$ mentioned above can be extended to $X'$ without additional zeros. Hence $\langle c_1(\L')^2,X' \rangle = [\Sigma]^2 = \langle c_1(\L)^2 ,X \rangle$.
\end{proof}

For a $1$-surgery along a homologically trivial loop, all $\text{spin}^\C$ structures can be extended to the new manifold. The extension is not unique. However, it would not change the index of the Dirac operator.
\begin{theorem}\label{thm:changeOfSpinc1}
Let $X$ be any $4$-manifold with $H^1(X;\Z) = 0$. Let $\gamma$ be a homologically trivial loop that we choose to do the surgery. Let $N = \S^1 \times D^3$ be a tubular neighborhood of $\gamma$ with radius sufficiently small.
    Let $X_0$ be the complement of $N$. Let $X'= X_0 \cup_{\S^1 \times \S^2} (D^2 \times \S^2)
    $ be the manifold obtained by doing the surgery on $X$ along $\gamma$. Let $\ss$ be any $\text{Spin}^\C$ structure over $X$ and $\L$ be the corresponding determinant line bundle. Let $\mathcal{S}(X')$ be the set of $\text{spin}^\C$ structures on $X'$, and
    \[
    \mathcal{S}(X', \ss) := \{\Gamma \in \mathcal{S}(X') ; \left.\Gamma\right|_{X_0} =\left. \ss\right|_{X_0}\}.
    \]
    Then $ \mathcal{S}(X', \ss) $ contains a $\Z$-family of $\text{Spin}^\C$ structures over $X'$,
    and for any $\ss' \in  \mathcal{S}(X', \ss) $ the determinant line bundle $\L'$ associated to $\ss'$ satisfies
    \begin{equation}\label{equ:c1-square}
    \langle c_1(\L')^2,X' \rangle = \langle c_1(\L)^2 ,X \rangle.
    \end{equation}
\end{theorem}

\begin{proof}
We first show that $\mathcal{S}(X', \ss)$ is nonempty. Let $\ss'$ be any $\text{Spin}^\C$ structure over $X'$. By Remark \ref{remark:spincline}, 
the difference between $\ss'|_{X_0}$ and $\ss|_{X_0}$ is a complex line bundle $L_0$ over $X_0$, namely, $\ss'|_{X_0} \otimes L_0 = \ss|_{X_0}$. 

We still have that $L_0$ can be extended to a complex line bundle over $X'$, but now the extension is not unique. Now the induced homomorphisms
\[
    i^*: H^2(X') \to H^2(X_0)
\]
is surjective and has kernel $\Z$. This follows from the following Mayer-Vietoris sequence:

\begin{center}
\begin{tikzpicture}[commutative diagrams/every diagram]
\node (P0) at (0cm, 0cm) {$H^1(X_0)\oplus H^1(D^2 \times \S^2)$};
\node (P1) at (-1cm, -0.1cm) {};
\node (P2) at (3.3cm, 0cm) {$H^1(\S^1 \times \S^2)$} ;
\node (P6) at (3.3cm, -1cm) {$\Z$} ;
\node (P3) at (5.5cm, 0cm) {$H^2(X')$};
\node (P4) at (-1cm, -1cm) {?};
\node (P5) at (8.5cm, 0cm) {$H^2(X_0)\oplus H^2(D^2 \times \S^2)$};
\node (P7) at (8.7cm, -0.1cm) {};
\node (P8) at (8.7cm, -1cm) {$\Z$};
\node (P9) at (12cm, 0cm) {$H^2(\S^1 \times \S^2)$};
\node (P10) at (12cm, -1cm) {$\Z$};
\path[commutative diagrams/.cd, every arrow, every label]
(P0) edge node {$0$} (P2)
(P2) edge node {} (P3)
(P1) edge node {$\cong$} (P4)
(P3) edge node {} (P5)
(P2) edge node {$\cong$} (P6)
(P4) edge node {$0$} (P6)
(P7) edge node {$\cong$} (P8)
(P5) edge node {} (P9)
(P8) edge node {$ i_\partial^*$} (P10)
(P9) edge node {$\cong$} (P10);
\end{tikzpicture}
\end{center}
For the surjectivity of $i^*$, pick any $\sigma \in H^2(X_0)$. Then there exists a $\tau \in H^2(D^2 \times \S^2)$ such that 
\begin{align*}
 H^2(X_0) \oplus  H^2(D^2 \times \S^2) &\to  H^2(\S^1 \times \S^2)\\
 (\sigma,\tau) &\mapsto 0
\end{align*}
since $ i_\partial^*$ is an isomorphism. Hence $(\sigma,\tau)$ is in the image of $H^2(X')\to H^2(X_0) \oplus  H^2(D^2 \times \S^2)$ and $\sigma$ is in the image of $i^*$. Choose any complex line bundle $L'$ on $X'$ such that $i^*(c_1(L')) = c_1(L_0)$. Then $\left.(\ss' \otimes L')\right|_{X_0} = \ss$.


To prove the last statement, let $\ss'$ be in $\mathcal{S}(X', \ss) $ and $\L'$ be its determinant line bundle. We want to show (\ref{equ:c1-square}). Topologically, as in the proof of Theorem \ref{thm:changeOfSpinc}, the dual of $c_1(\L)$ is a surface $\Sigma \subset X_0$. But now in addition to $\Sigma$, the dual of $c_1(\L')$ may contain some spheres in $D^2 \times \S^2$ , which come from $\ker i^*$.

To understand the kernel of $i^*$, we have to understand the image of the connecting homomorphism $H^1(\S^1 \times \S^2) \to H^2(X')$. The loop $\gamma$ is homologically trivial, so $\gamma$ bounds a surface in $X_0$. The union of this surface and the core of $D^2 \times \S^2$ is a closed surface $F$ in $X'$. The generator of $\ker i^*$ sends $[F]$ to $1$. So its dual is represented by $\{pt\} \times \S^2 \subset D^2\times \S^2 \subset X'$. The self-intersection of this submanifold in $X'$ must be zero, whatever the framing of the surgery is (for example, the surgery on a homotopically trivial loop in a spin manifold may produce a connected summand $\S^2 {\times} \S^2$ or $\S^2\tilde{\times} \S^2$). Since $\{pt\} \times \S^2 \subset D^2\times \S^2$ doesn't intersect $\Sigma \subset X_0$,
\begin{equation*}
    \langle c_1(\L')^2,X' \rangle =[\Sigma] \cap [\Sigma] + [\{pt\} \times \S^2 ] \cap  [\{pt\} \times \S^2 ] =[\Sigma] \cap [\Sigma] =\langle c_1(\L)^2 ,X \rangle.
\end{equation*}

\end{proof}
\begin{example}
Let $X = \S^4$ and the surgery over a Lie framed circle produces $X' = \S^2\widetilde{\times} \S^2$. The Kirby diagram for $X_0 \cong D^2 \times \S^2$ is a $0$-framed circle. To glue a $D^2 \times \S^2$ to $X_0$ is to add a $1$-framed circle that links to the $0$-framed circle. Then above $F$ is the sphere as the union of the disk spanned by the $1$-framed circle and the core of the $2$-handle attaching to it. Above $\{pt\} \times \S^2$ is the sphere corresponding to the $0$-framed circle.

Note that $F$ is the exceptional sphere of $\C P^2$. Let $G$ be the exceptional sphere of $\overline{\C P}^2$. Let $x$ and $y$ be the dual classes of $F$ and $G$. The intersection form for the basis $\{x,y\}$ is 
\begin{equation*}
\begin{bmatrix}
1 & 0 \\[10pt]
0 & -1
\end{bmatrix}
\end{equation*}
$F$ and $G$ correspond to unlinked $1$-framed circle and $-1$-framed circle in the Kirby diagram of $\C P^2 \# \overline{\C P}^2$. The handle slide of the $-1$-framed circle around the $1$-framed circle produces the $0$-framed circle in the Kirby diagram of $\S^2\widetilde{\times} \S^2$ (see Figure \ref{fig:kirby}). Hence $x +y$ is represented by above $\{pt\} \times \S^2$, and is a generator of $\ker i^*$. $\ker i^*$ contains all complex line bundles of $X'$ that come from $X$.
\begin{figure}[ht!]
    \begin{Overpic}{\includegraphics[scale=0.7]{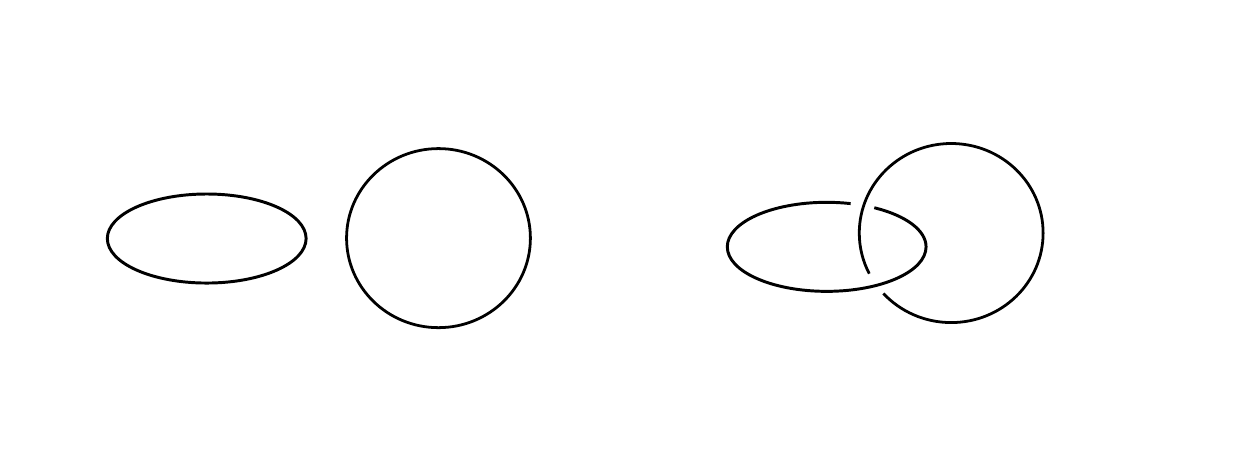}}
        \put(20,12){$1$}
        \put(38,9){$-1$}
        \put(50,17){$\to$}
        \put(64,11){$1$}
        \put(80,9){$0$}
    \end{Overpic}
    \caption{}
    \label{fig:kirby}
\end{figure}

Since $w_2(X') = x +y$, all $\text{spin}^\C$ structures on $X'$ have determinant line bundles $mx + ny$ with $m,n$ odd. Let $\ss$ be the unique $\text{spin}^\C$ structures on $X$. Then $\mathcal{S}(X', \ss)$ are $\text{spin}^\C$ structures on $X'$ with determinant line bundles $a(x + y)$ with $a$ odd. The following facts illustrate some points in the proof of Theorem \ref{thm:changeOfSpinc1}:
\begin{itemize}
\item The extension of $\ss$ to $X'$ is not unique;
\item $\mathcal{S}(X', \ss)$ doesn't contain all $\text{spin}^\C$ structures on $X'$;
\item For any $\ss' \in \mathcal{S}(X', \ss)$, $c_1(\ss')-c_1(\ss) $ are some copies of spheres $(x+y)$, and $(x+y)^2=0$;
\item These spheres would not intersect $c_1(\ss)$ even when $X$ is nontrivial, because $(x+y)$ only intersects $x$ and $y$, which are introduced by the surgery and not in $H^2(X)$.
\end{itemize}

\end{example}

In the gluing theory of Seiberg-Witten monopoles, the Seiberg-Witten equations and thus the $\text{spin}^c$ structure of the boundary $\partial X_0 = \partial N = \S^1\times\S^2$ would be considered. Hence one has to consider how to restrict the $\text{spin}^c$ structure of the $4$-manifold $X_0$ to the $3$-manifold $\S^1\times\S^2$.

Let $X$ be any $4$-manifold with boundary $\partial X$. Identify $TX|_{\partial X}$ with $T\partial X \oplus \nu$ where $\nu$ is the normal bundle of $\partial X\subset X$. 
Let $P_{SO(4)}$, $P_{SO(3)}$ be the frame bundles of $T\partial X \oplus \nu$ and $T\partial X$, Let $g^{(4)}\in SO(4)$, $g^{(3)}\in SO(3)$ be corresponding transition functions on a point $x\in \partial X$. The following diagram commutes:
\begin{center}
\begin{tikzpicture}[commutative diagrams/every diagram]
\node (P0) at (0cm, 0cm) {Fr(3)};
\node (P1) at (0cm,  -1.6cm) {Fr(3)};
\node (P2) at (1.7cm, 0cm) {Fr(4)};
\node (P3) at (1.7cm, -1.6cm) {Fr(4)};
\node (P4) at (0.4cm, -0.8cm) {$g^{(3)}$};
\node (P5) at (1.4cm, -0.8cm) {$g^{(4)}$};
\node (P5) at (0.9cm, -0.9cm) {$\mapsto$};
\node (P6) at (0.9cm, -0.7cm) {$i$};
\path[commutative diagrams/.cd, every arrow, every label]
(P0) edge node {} (P1)
(P0) edge node {} (P2)
(P1) edge node {} (P3)
(P2) edge node {} (P3);
\end{tikzpicture}
\end{center}
where the top and bottom horizontal arrows are given by adding an inner vector. Then the map $i$ between transition functions is given by the natural embedding of $SO(3)\to SO(4)$.

Let $\mathbb H$ be quaternions and $\SU(2)=\S^3$ be the group of unit quaternions. $q\in \SU(2)$ acts on $\imaginary\mathbb H$ by
\[
x \mapsto qxq^{-1},
\]
which gives the double cover $\rho_3:\SU(2) = \text{Spin}(3) \to \SO(3)$. $(p,q)\in \SU(2)  \times \SU(2) =\text{Spin}(4) $ acts on $\mathbb H$ by
\[
x \mapsto pxq^{-1},
\]
which gives the double cover $\rho :\text{Spin}(4) \to \SO(4)$. Regard the real axis of $\mathbb{H}$ as the normal space of $x\in \partial X$, then 
\begin{align*}
i: \text{Spin}(3)  &\to \text{Spin}(4) \\
q&\mapsto (q,q)
\end{align*}
covers the embedding $i:SO(3)\to SO(4)$. Similarly we have a map
\begin{align*}
i^c: \text{Spin}^c(3):=S^1 \times  \text{Spin}(3) / \{\pm (1,I)\} &\to \text{Spin}^c(4)\\
[z,q]&\mapsto [z,q,q]
\end{align*}
that covers $i:SO(3)\to SO(4)$. Hence a $\text{spin}$($\text{spin}^c$) structure of $X$ induces a $\text{spin}$($\text{spin}^c$) structure of $\partial X$. Moreover, from the definition of $i^c$, the restriction of a $\text{spin}^c$ structure is compatible with the restriction of its determinant line bundle.

\begin{prop}\label{prop:spincRestriction}
Use the notations in Theorem \ref{thm:changeOfSpinc}. Then $\ss|_{\partial X_0}$ is the only $\text{spin}^c$ structure of $ \S^1 \times \S^2 $ such that the first Chern class of the determinant line bundle is zero, and $\ss'|_{D^2 \times \S^2}$ is the only $\text{spin}^c$ structure of $ D^2 \times \S^2 $ such that the first Chern class of the determinant line bundle is zero.
\end{prop}
\begin{proof}
$\det(\ss|_{\partial X_0}) = \det(\ss)|_{\partial X_0}$ is the restriction of the trivial line bundle $ \det(\ss)|_{\S^1 \times D^3}$.  So $\det(\ss|_{\partial X_0})$ is trivial. $H^2(\partial X_0;\Z) = 0$ so by Remark \ref{remark:spincline} $\ss|_{\partial X_0}$ is the only $\text{spin}^c$ structure of $ \S^1 \times \S^2 $.

$\det(\ss'|_{D^2 \times \S^2})|_{\partial (D^2 \times \S^2)}=\det(\ss)|_{\S^1 \times \S^2} $ is trivial. Since the restriction $H^2(D^2 \times \S^2) \cong H^2( \S^1 \times\S^2) $ is an isomorphism, $c_1(\det(\ss'|_{D^2 \times \S^2})) = 0$. 
$H^2(D^2 \times \S^2;\Z)$ has no torsion so by Remark \ref{remark:spincline} $c_1\circ \det$ is injective. Hence $\ss'|_{D^2 \times \S^2}$ is the only $\text{spin}^c$ structure of $D^2 \times \S^2$ such that the first Chern class of the determinant line bundle is zero.
\end{proof}




\subsection{Seiberg-Witten equation, transversality results, and ASD operator}

\subsection{Positive scalar curvature}\label{subsection:PSCmetric}

A positive scalar curvature will give two desired properties: First, by the Weitzenb{\"o}ck formula, a non-negative scalar curvature on $3$- or $4$-manifolds leads solely to reducible solutions of the Seiberg-Witten equation (see \cite{kronheimer_mrowka_2007} (4.22)). Second, by the Weitzenb{\"o}ck formula and integration by parts, we have (see page 105 of \cite{Nicolaescu2000NotesOS})
\[
\int_M|\D_A\psi|^2dv_g = \int_M (|\nabla^A\psi|^2  + \frac{s}{4}|\psi|^2 +\frac{1}{2}\langle {\mathbf{c}}(F_A^+)\psi,\psi\rangle ) dv_g 
\]
where $A$ is a connection, $\D_A$ is the twisted Dirac operator, $s$ is the scalar curvature, and $\mathbf{c}$ is Clifford multiplication. So if $s$ is everywhere positive and $A$ is flat, the twisted Dirac operator would have trivial kernel.

It turns out that we can construct bullet metrics on $\S^1 \times D^3$ and $D^2\times \S^2$ such that the corresponding Levi-Civita connections have positive scalar curvature everywhere.

To construct the bullet metric on $D^2 \times S^2$, embed it in $\R^3\times \R^3$ such that the component $\S^2$ is standard sphere, and $D^2$ is the union of a standard semi-sphere $\S^2_+$ and a cylinder $\partial D^2 \times I$, which is the collar neighborhood of $\partial D^2$. One can perturb this  embedding to make it smooth, and the metric $g$ of $D^2\times S^2$ induced by the standard metric of $\R^3\times \R^3$ is so-called bullet metric.

One can compute the scalar curvature of this metric using the following formula:
\[
s= \sum_{i\neq j}\mathsf{sec}(e_i,e_j)
\]
where $\mathsf{sec}$ is the sectional curvature and $\{e_i\}$ is a set of orthonormal basis. The sectional curvature of $\S^2$ and $\S^2_+$ is positive. If two vectors lie in different copies of $\R^3$ in $\R^3\times \R^3$, the sectional curvature of the plane identified by these vectors is zero. This means that
\[
s({D^2\times S^2})= s({D^2}) + s({S^2}).
\]
Therefore, the scalar curvature is everywhere positive.

For $\S^1 \times D^3$, embed it in  $\R^2\times \R^4$ such that $\S^1$ is standard circle and $D^3$ is the union of a standard semi-sphere $\S^3_+$ and a cylinder $\partial D^3 \times I$. By the same reasoning and the fact that $\partial D^3=\S^3$ also has positive scalar curvature, the scalar curvature of $\S^1 \times D^3$ is everywhere positive.

\section{Apply ordinary gluing theory to $1$-surgery}
In ordinary gluing theory, one obtain the union ${N}_r$ of two manifolds ${N}_1$ and ${N}_2$ by gluing along their boundaries $N$, and consider the relation between monopoles over ${N}_1$ and ${N}_2$ and monopoles over the union ${N}_r$.

Given a pair of monopoles on ${N}_1$ and ${N}_2$, respectively, if they are compatible over boundaries, one can glue them to obtain a point of configuration space over the union ${N}_r$. It turns out that there exists a genuine monopole of ${N}_r$ near this point. Moreover, the space of genuine monopoles over the union ${N}_r$ is actually isotopic to the manifold of configurations obtained by gluing in this way.
 
 The proof of the global gluing theorem is divided to four steps: The \textbf{linear gluing theorem} will give an approximation of the kernel of boundary difference map. The \textbf{local gluing theorem} will describe the set of genuine monopoles in a neighborhood of each glued configuration point. The \textbf{local surjectivity theorem} will prove that, the set of such neighborhoods is a cover of the manifold of genuine monopoles. The \textbf{global gluing theorem} will prove that, the moduli space of genuine monopoles is homeomorphic to the moduli space of glued configuration points, if the obstruction space is trivial.
 
 In this section, we will follow the strategy in Nicolaescu's book \cite{Nicolaescu2000NotesOS}.  In our case, i.e, $N = \S^1 \times \S^2$, $N_2 = \S^1 \times D^3$ or $D^2\times \S^2$, one can just apply the linear gluing theorem and the local surjectivity theorem in charpter 4 of \cite{Nicolaescu2000NotesOS}, and prove the condition of the local gluing theorem is satisfied. However, the global gluing theorem in this situation is slightly different from what Nicolaescu presented.

\subsection{Abstract linear gluing results}
In this subsection, we review the abstract linear gluing results in section 4.1 of \cite{Nicolaescu2000NotesOS}.

It's natural to expect that, a longer neck of ${N}_r$ will narrow the difference between genuine monopoles and configurations obtained by gluing, since there should be no difference when the length of the neck $r=\infty$. So we first consider manifolds with necks of infinite length, say, $\hat{N}_1= {N}_1\cup_N N\times [0,\infty)$ and  $\hat{N}_2= {N}_2\cup_N N\times [0,\infty)$. Such manifolds are called cylindrical manifolds.

Suppose $\beta(t)$ is a smooth cutoff function such that $\beta(t)=0$ on $(-\infty,1/2]$ and $\beta(t) = 1$ on $[1,\infty)$. Set $\alpha(t)=1-\beta(t)$. These functions will be used to glue a pair of sections.

Denote by $\hat{E}$ a cylindrical bundle over a cylindrical manifold $\hat{N}$, that is, a vector bundle $\hat{E} \to \hat{N}$ together with a vector bundle ${E} \to {N}$ and a bundle isomorphism
\[
\hat{E}|_{N\times [0,\infty)} \to \pi^*E,
\]
where $ \pi: N\times [0,\infty) \to N$ is the projection map. Let $L^p(\hat{E})$ be the space of $L^p$-sections of $\hat{E}$. 
Let $L^p_{loc}(\hat{E})$ be the space of measurable sections $u$ such that $u\varphi \in L^p(\hat{E})$ for any smooth, compactly supported function $\varphi$ on $\hat{N}$. 
Denote by $\hat{u}$ an $L^2_{loc}$-section of $\hat{E}$. If there exists an $L^2_{loc}$-cylindrical section $\hat{u}_0$ such that
\[
\hat{u}-\hat{u}_0 \in L^2(\hat{E}),
\]
then $\hat{u}$ is called \textbf{asymptotically cylindrical} (or \textbf{a-cylindrical}). Define the asymptotic value of $\hat{u}$ to be
\[
\partial_\infty \hat{u} := \partial_\infty \hat{u}_0.
\]
Let $ L^2_\mu(\hat{E}) = \{ u\in  L^2(\hat{E}); \| u|_{\hat{N} \setminus N\times [0,\infty)}\|_{L^2} + \| u|_{N\times [0,\infty)} \cdot \e^{\mu t} \|_{L^2} < \infty\}$. The supremum of all $\mu \geq 0$ such that 
\[
\hat{u} -\hat{u}_0 \in  L^2_\mu(\hat{E})
\] 
is called the \textbf{decay rate} of the a-cylindrical section $\hat{u}$.

The norm on the space of a-cylindrical sections is defined by
\[
\|\hat{u}\|_{ex} = \|\hat{u}-\hat{u}_0\|_{L^2} + \|\partial_\infty \hat{u}\|_{L^2}
\]
The resulting Hilbert space is called $L_{ex}^2$.


Given a pair of compatible cylindrical sections $\hat{u}_i$ of $\hat{E}_i$, i.e they share the same constant value over the neck, they can be glued to form a section $ \hat{u}_1  \#_r  \hat{u}_2$ of $\hat{E}_1  \#_r  \hat{E}_2$. If $\hat{u}_i$ are just compatible $L_{ex}^2$-sections, i.e they are a-cylindrical sections with identical asymptotic values $\partial_\infty \hat{u}_1 = \partial_\infty \hat{u}_2$, they should be modified by cutoff functions first. Let $\hat{u}_i(r)$ be the same section as $\hat{u}_i$ outside the neck, and on the neck
\begin{equation}\label{equ:cut}
\hat{u}_i(r)(t) = \alpha(t-r)\hat{u}_i + \beta(t-r)\partial_\infty \hat{u}_i.
\end{equation}
When $t<r$, $\hat{u}_i(r) = \hat{u}_i$, and when $t>r+1$, $\hat{u}_i(r)$ is just the asymptotic value of $ \hat{u}_i$. Thus $ \hat{u}_i(r)$ is an approximation of $ \hat{u}_i$ as $r\to \infty$. Now these genuine cylindrical sections can be glued along the neck, so we define
\begin{equation}\label{equ:paste}
\hat{u}_1  \#_r  \hat{u}_2 := \hat{u}_1(r)  \#_r  \hat{u}_2(r)
\end{equation}

In the following description, all verifications of smoothness, Fredholmness and exactness are obmitted. See Section 4.3 of Nicolaescu's book for details.

Let $L^{m,p}$ be the space of sections with finite Sobolev norm $\|\cdot\|_{m,p}$.
Let $\hat{\sigma}$ be a $\text{spin}^c$ structure of $\hat{N}$ such that it induces a $\text{spin}^c$ structure $\sigma$ of $N$. Denote by $\mathcal{C}_\sigma$ the space of configurations in $L^{2,2}$ over the $3$-manifold $N$, by
\[
\mathcal{Z}_\sigma \subset \mathcal{C}_\sigma
\]
the set of monopoles (solutions of Seiberg-Witten equations) on $N$, and by
\[
\mathfrak{M}_\sigma = \mathcal{Z}_\sigma/ \mathcal{G}_\sigma
\]
the moduli space of monopoles on $N$.

Define
\begin{equation}\label{equ:defCsw}
\hat{\mathcal{C}}_{\mu ,sw} := \partial_\infty^{-1}(\mathcal{Z}_\sigma)
\end{equation}
and 
\[
\hat{\mathcal{Y}}_\mu := L_\mu^{1,2}(\hat{S}_{\hat{\sigma}}^- \oplus \mathbf{i}\Lambda_+^2T^*\hat{N}).
\]
The Seiberg-Witten equations give the Seiberg-Witten map
\begin{align*}
\widehat{SW}: \hat{\mathcal{C}}_{\mu ,sw} &\to \hat{\mathcal{Y}}_\mu,\\
(\hat{\psi},\hat{A}) &\mapsto \D_{\hat{A}}\hat{\psi}\oplus (\sqrt{2}(F^+_{\hat{A}} - \frac{1}{2}\hat{\mathbf{c}}^{-1}(q(\hat{\psi}))),
\end{align*}
where $\D_{\hat{A}}$ is the Dirac operator twisted by the connection $\hat{A}$, and $ \hat{\mathbf{c}}$ is the Clifford multiplication on $\hat{N}$.

We will use the following notation:
\[
\widehat{\mathcal{G}}_{\mu,ex} := \{ \hat{u} \in L^{ 3,2}_{\mu,ex}(\hat{N},\C); | \hat{u}(p)|=1\text{  } \forall p\in \hat{N}\}
\]
\[
\widehat{\M}_\mu := \widehat{SW}^{-1}(0) / \widehat{\mathcal{G}}_{\mu,ex}.
\]

$\hat{\mathsf{C}}_0=(\hat{\psi}_0 , \hat{A}_0)$: A fixed smooth finite energy monopole on $\hat{N}$. $\hat{\mathsf{C}}_0$ modulo a gauge transformation is in $\hat{\mathcal{C}}_{\mu ,sw}$  (see section 4.2.4 of Nicolaescu's book \cite{Nicolaescu2000NotesOS}). So in this paper we always assume that $\hat{\mathsf{C}}_0 \in \hat{\mathcal{C}}_{\mu ,sw}$.

${\mathsf{C}}_\infty$: A fixed smooth finite energy monopole on ${N}$.

$\widehat{\underline{SW}}_{\hat{\mathsf{C}}_0}$: The linearization of $\widehat{SW}$ at $\hat{\mathsf{C}}_0$.

As a Lie group, the component of $\mathbf{1}$ of $\hat{\mathcal{G}}_{\mu,ex}$ consists of elements that can be written as $\e^{\mathbf{i}f}$ where $f\in  L_{\mu,ex}^{3,2}(\hat{N},\mathbf{i}\R)$. Recall that we have fixed $\hat{\mathsf{C}}_0$, so the gauge action gives a map
\begin{align*}
\hat{\mathcal{G}}_{\mu,ex} &\to \hat{\mathcal{C}}_{\mu,sw} \\
\hat{u} &\mapsto \hat{u}  \cdot \hat{\mathsf{C}}_0.
\end{align*}
Denote the stabilizer of $\hat{\mathsf{C}}_0$ under the gauge action by $\hat{G}_0$. The differential of the above map is 
\begin{align*}
\mathfrak{L}_{\hat{\mathsf{C}}_0} :T_{\mathbf{1}}\hat{\mathcal{G}}_{\mu,ex} &\to T_{\hat{\mathsf{C}}_0}\hat{\mathcal{C}}_{\mu,sw}\\
\mathbf{i}f &\mapsto (\mathbf{i}f\hat{\psi}_0,-2\mathbf{i} df)
\end{align*}

We have three differential complexes:
\[\label{equ:complexF}
0\to T_1 \hat{\mathcal{G}}_{\mu} 
\xrightarrow{\mathfrak{L}_{\hat{\mathsf{C}}_0}} T_{\hat{\mathsf{C}}_0} \partial_\infty^{-1}(\mathsf{C}_\infty) 
\xrightarrow{\widehat{\underline{SW}}_{\hat{\mathsf{C}}_0}} T_0\mathcal{Y}_\mu
\to 0
\tag{$F_{\hat{\mathsf{C}}_0}$}
\]
\[\label{equ:complexK}
0
\to T_1 \hat{\mathcal{G}}_{\mu,ex} 
\xrightarrow{\frac{1}{2}\mathfrak{L}_{\hat{\mathsf{C}}_0}} T_{\hat{\mathsf{C}}_0} \hat{\mathcal{C}}_{\mu,sw} 
\xrightarrow{\widehat{\underline{SW}}_{\hat{\mathsf{C}}_0}} T_0\mathcal{Y}_\mu
\to 0
\tag{$\widehat{\mathcal{K}}_{\hat{\mathsf{C}}_0}$}
\]
\[\label{equ:complexB}
0\to T_1\mathcal{G}_\sigma 
\xrightarrow{\frac{1}{2}\mathfrak{L}_{{\mathsf{C}}_\infty}} T_{\mathsf{C}_\infty}\mathcal{Z}_\sigma
\to 0
\to 0
\tag{$B_{\hat{\mathsf{C}}_0}$}
\]
In the category of differential complexes, it's easy to verify that
\begin{equation}\label{equ:shortExact}
0\to \text{\ref{equ:complexF}}
\stackrel{i}{\hookrightarrow} \text{\ref{equ:complexK}} 
\stackrel{\partial_\infty}{\twoheadrightarrow} \text{\ref{equ:complexB}} 
\to 0
\tag{\textbf E}
\end{equation}
is an exact sequence. Namely, each column of the diagram 
\begin{equation}\label{equ:diagram} 
\xymatrix{
&
0\ar[d] &
0\ar[d] & 
0 \ar[d]& 
\\
0\ar[r] &
T_1 \hat{\mathcal{G}}_{\mu}  \ar[d] \ar[r]^(0.4){\frac{1}{2}\mathfrak{L}_{\hat{\mathsf{C}}_0}} &
T_{\hat{\mathsf{C}}_0} \partial_\infty^{-1}(\mathsf{C}_\infty) \ar[d] \ar[r]^(0.6){\widehat{\underline{SW}}_{\hat{\mathsf{C}}_0}} &
T_0\mathcal{Y}_\mu \ar[d]^{=}\ar[r] & 
0\\
0\ar[r] &
T_1 \hat{\mathcal{G}}_{\mu,ex}  \ar[d] \ar[r]^(0.5){\frac{1}{2}\mathfrak{L}_{\hat{\mathsf{C}}_0}} &
T_{\hat{\mathsf{C}}_0} \hat{\mathcal{C}}_{\mu,sw}  \ar[d]^{\partial_\infty^0} \ar[r]^(0.5){\widehat{\underline{SW}}_{\hat{\mathsf{C}}_0}} &
T_0\mathcal{Y}_\mu \ar[d]\ar[r] &
 0 \\
0\ar[r] &
T_1\mathcal{G}_\sigma  \ar[d] \ar[r]^(0.4){\frac{1}{2}\mathfrak{L}_{{\mathsf{C}}_\infty}} &
T_{\mathsf{C}_\infty}\mathcal{Z}_\sigma \ar[d] \ar[r]&
0 \ar[d]\ar[r] & 
0 \\
&
0 &
0& 
0& 
\\
}
\tag{\textbf{D}}
\end{equation}
is exact.
Set 
\[
H^i_{\hat{\mathsf{C}}_0} := H^i( \widehat{\mathcal{K}}_{\hat{\mathsf{C}}_0}).
\]
For $i=0$, observe that 
\[
H^0_{\hat{\mathsf{C}}_0}  \cong T_1\hat{G}_0
\]
is the tangent space of the stabilizer of $\hat{\mathsf{C}}_0$ under gauge action. It is one dimensional if $\hat{\mathsf{C}}_0$ is reducible and trivial otherwise. For $i=1$, observe that $\dim_\R(H^1_{\hat{\mathsf{C}}_0})$ is the dimension of the formal tangent space of $\widehat{\M}_\mu$ at $[\hat{\mathsf{C}}_0]$. For $i=2$, $H^2_{\hat{\mathsf{C}}_0}$ is called the \textbf{obstruction space} at $\hat{\mathsf{C}}_0$.

From the diagram \ref{equ:diagram} 
we obtain a long exact sequece
\begin{equation}\label{equ:longExact}
\begin{tikzcd}
& H^0(F_{\hat{\mathsf{C}}_0}) \arrow[d]
& H^1(F_{\hat{\mathsf{C}}_0}) \arrow[d]
& H^2(F_{\hat{\mathsf{C}}_0}) \arrow[d] & \\
\arrow[r, phantom, ""{coordinate, name=Y}] & H^0_{\hat{\mathsf{C}}_0}  \arrow[d]\arrow[r, phantom, ""{coordinate, name=Z}]
&  H^1_{\hat{\mathsf{C}}_0}  \arrow[d]\arrow[r, phantom, ""{coordinate, name=T}]
& H^2_{\hat{\mathsf{C}}_0} \arrow[d] &\\
0   \arrow[ruu,
"",
rounded corners,
to path={
-| (Y) [near end]\tikztonodes
|-  (\tikztotarget)}]  
& H^0(B_{\hat{\mathsf{C}}_0})  \arrow[ruu,
"",
rounded corners,
to path={
-| (Z) [near end]\tikztonodes
|-  (\tikztotarget)}]  
& H^1(B_{\hat{\mathsf{C}}_0})\arrow[ruu,
"",
rounded corners,
to path={
-| (T) [near end]\tikztonodes
|-  (\tikztotarget)}]  
& 0 & 
\end{tikzcd}
\tag{\textbf L}
\end{equation}

$\hat{\mathsf{C}}_0$ is called \textbf{regular} if $H^2_{\hat{\mathsf{C}}_0}= 0$, and \textbf{strongly regular} if $H^2(F_{\hat{\mathsf{C}}_0}) = 0$. Note that by the long exact sequance, strong regularity implies regularity.

The integer
\[
d(\hat{\mathsf{C}}_0) := -\chi(\widehat{\mathcal{K}}_{\hat{\mathsf{C}}_0}) =- \dim_\R H^0_{\hat{\mathsf{C}}_0}+  \dim_\R H^1_{\hat{\mathsf{C}}_0} - \dim_\R H^2_{\hat{\mathsf{C}}_0}
\]
is called the \textbf{virtual dimension} at $[\hat{\mathsf{C}}_0]$ of the moduli space $\widehat{\M}_\mu$. If $\hat{\mathsf{C}}_0$ is regular irreducible, $\widehat{\M}_\mu$ is smooth at $\hat{\mathsf{C}}_0$, and
\[
d(\hat{\mathsf{C}}_0)=-0+ \dim_\R H^1_{\hat{\mathsf{C}}_0}-0
\]
is indeed the dimension of the tangent space of $\widehat{\M}_\mu$ at $[\hat{\mathsf{C}}_0]$. On the other hand, if $\hat{\mathsf{C}}_0$ is regular reducible, we have 
\[
d(\hat{\mathsf{C}}_0)=-1+ \dim_\R H^1_{\hat{\mathsf{C}}_0}-0
\]
So $\dim_\R H^1_{\hat{\mathsf{C}}_0} = d(\hat{\mathsf{C}}_0)+1$. The difference between irreducibles and reducibles, comes from the fact that the orbit of irreducible $\hat{\mathsf{C}}_0$ is $1$-dimensional in $\hat{\mathcal{C}}_{\mu,sw}$, given by the action of constant gauge, while the constant gauge acs on reducibles trivially.

The $L^2_\mu$-adjoint of $\mathfrak{L}_{\hat{\mathsf{C}}_0}$ is
\begin{equation}\label{equ:L*}
\mathfrak{L}^{*_\mu}_{\hat{\mathsf{C}}_0} : (\dot\psi ,\mathbf{i}\dot a) \mapsto -2
\mathbf{i} d^{*_\mu} \dot a - \mathbf{i} \imaginary \langle\psi,\dot \psi\rangle_\mu.
\end{equation}
Now define 
\[
\hat{\mathcal{T}}_{\hat{\mathsf{C}}_0,\mu} := \widehat{\underline{SW}}_{\hat{\mathsf{C}}_0} \oplus \frac{1}{2}\mathfrak{L}^{*_\mu}_{\hat{\mathsf{C}}_0} :  L_\mu^{2,2}(\hat{\S}_{\hat{\sigma}}^+ \oplus \mathbf{i}T^*\hat{N}) \to \hat{\mathcal{Y}}_\mu \oplus L_\mu^{1,2}(N,\mathbf{i}\R).
\]
We can deduce that (see the proof of Lemma 4.3.19 of Nicolaescu's book)
\begin{equation}\label{equ:Tinf}
\vec{\partial}_\infty\hat{\mathcal{T}}_{\hat{\mathsf{C}}_0,\mu} =
\mathcal{T}_{{\mathsf{C}}_\infty,\mu } =
\begin{bmatrix}
\underline{SW}_{{\mathsf{C}}_\infty}  & -\frac{1}{2}\mathfrak{L}_{{\mathsf{C}}_\infty} \\[10pt]
\frac{1}{2}\mathfrak{L}^{*}_{{\mathsf{C}}_\infty} & -2\mu
\end{bmatrix}
\end{equation}
It turns out that we can remove the dependence on the choice of $\mu$, such that everything is independant of $\mu$ (Page 387 of \cite{Nicolaescu2000NotesOS}). Set $\mu = 0$ formally:
\begin{equation}\label{defOfTWithoutMu}
\mathcal{T}_{\hat{\mathsf{C}}_0} := \widehat{\underline{SW}}_{\hat{\mathsf{C}}_0} \oplus \frac{1}{2}\mathfrak{L}^*_{\hat{\mathsf{C}}_0}
\end{equation}
From the description \ref{equ:Tinf} above of $\mathcal{T}_{{\mathsf{C}}_\infty,\mu }$ ($\mu=0$), we have a decomposition
\[
\ker\mathcal{T}_{{\mathsf{C}}_\infty} = T_{{\mathsf{C}}_\infty}\M_\sigma \oplus T_1G_\infty,
\]
where $G_\infty$ is the stabilizer of $\hat{\mathsf{C}}_\infty$ under gauge action.
Denote the two components of the boundary map
\[
\partial_\infty: \ker_{ex}\hat{\mathcal{T}}_{\hat{\mathsf{C}}_0} \to \ker\mathcal{T}_{{\mathsf{C}}_\infty} = T_{{\mathsf{C}}_\infty}\M_\sigma \oplus T_1G_\infty
\]
by
\begin{align*}
\partial_\infty^0:  \ker_{ex}\hat{\mathcal{T}}_{\hat{\mathsf{C}}_0} &\to T_1G_\infty \\
\partial_\infty^c:  \ker_{ex}\hat{\mathcal{T}}_{\hat{\mathsf{C}}_0} &\to  T_{{\mathsf{C}}_\infty}\M_\sigma.
\end{align*}
Explictly, for $(\hat{\psi},\hat{\alpha}) \in L^{2,2}_{ex}(\hat{\S}_{\hat{\sigma}}^+ \oplus \mathbf{i}T^*\hat{N})$, if $\hat \alpha =\mathbf i \alpha + \mathbf i fdt$ on the neck $\R \times {N}$, where $\alpha(t)$ is a $1$-form on $N$ for each $t$, then
\begin{align}
\partial_\infty^0(\hat{\psi},\hat{\alpha}) &= \mathbf i\partial_\infty f \in T_1G_\infty \label{equ:partial0}\\
\partial_\infty^c(\hat{\psi},\hat{\alpha}) &=  (\partial_\infty \hat{\psi},{\partial_\infty \alpha}) \in T_{{\mathsf{C}}_\infty}\M_\sigma.
\end{align}

\subsection{Local gluing theorem}
Now we discuss how to apply the results in section 4.5 of Nicolaescu's book \cite{Nicolaescu2000NotesOS} to our cases.

Let's define
\[
\mathfrak{X}^k_+:= L^{k,2}(\hat{\S}^+_{\hat \sigma} \oplus \mathbf{i} T^*\hat{N}(r)), 
\mathfrak{X}^k_-:= L^{k,2}(\hat{\S}^-_{\hat \sigma} \oplus \mathbf{i} \Lambda^2_+T^*\hat{N}(r)), 
\]
\[
\mathfrak{X}^k := \mathfrak{X}^k_+ \oplus \mathfrak{X}^k_-.
\]
Define
\[
\hat{L}_r:= 
\begin{bmatrix}
0 &\hat{\mathcal T}^*_r \\
\hat{\mathcal T}_r & 0
\end{bmatrix}
:\mathfrak{X}^0 \to  \mathfrak{X}^0.
\]
We want to use the eigenspace corresponds to very small eigenvalues to approximate the kernel of this operator. Let $\H_r$ be the subspace of $\mathfrak{X}^0$ spanned by
\[
\{ v; \hat{L}_r v = \lambda v, |\lambda| < r^{-2}\}.
\]
Let $\mathcal{Y}_r$ be the orthogonal complement of $\H_r$ in $\mathfrak{X}^0$. Let $\H_r^\pm$ be the orthogonal projection of $\H_r$ to $\mathfrak{X}^0_\pm$.  Let $\mathcal{Y}_r^\pm$ be the orthogonal projection of $\mathcal{Y}_r$ to $\mathfrak{X}^0_\pm$.

Each row and column of the following diagrams is asymptotically exact (see page 434 of Nicolaescu's book \cite{Nicolaescu2000NotesOS}).

Virtual tangent space diagram:
\begin{equation}\label{equ:3T} 
\xymatrix{
&
0\ar[d] &
0\ar[d] & 
0 \ar[d]& 
\\
0\ar[r] &
\ker\Delta_+^c \ar[d] \ar[r]^{S_r} &
H^1_{\hat{\mathsf{C}}_1}\oplus H^1_{\hat{\mathsf{C}}_2}\ar[d] \ar[r]^{\Delta_+^c} &
L_1^+ + L_2^+ \ar[d]\ar[r] & 
0\\
0\ar[r] &
{\H}_r^+ \ar[d] \ar[r]^(0.3){S_r} &
\ker_{ex}\hat{\mathcal{T}}_{\hat{\mathsf{C}}_1}\oplus \ker_{ex}\hat{\mathcal{T}}_{\hat{\mathsf{C}}_2} \ar[d]^{\partial_\infty^0} \ar[r]^(0.64){\Delta_+^c} &
\hat{L}_1^+ + \hat{L}_2^+ \ar[d]\ar[r] &
 0 \\
0\ar[r] &
\ker\Delta_+^0 \ar[d] \ar[r]^{S_r} &
\mathfrak{C}_1^+\oplus \mathfrak{C}_2^+ \ar[d] \ar[r]^{\Delta_+^0} &
\mathfrak{C}_1^+ + \mathfrak{C}_2^+ \ar[d]\ar[r] & 
0 \\
&
0 &
0& 
0& 
\\
}
\tag{\textbf{T}}
\end{equation}

Obstruction space diagram:
\begin{equation}\label{equ:3O} 
\xymatrix{
&
0\ar[d] &
0\ar[d] & 
0 \ar[d]& 
\\
0\ar[r] &
\ker\Delta_-^c \ar[d] \ar[r]^(0.35){S_r} &
H^2(F_{\hat{\mathsf{C}}_1}) \oplus H^2(F_{\hat{\mathsf{C}}_2}) \ar[d] \ar[r]^(0.62){\Delta_-^c} &
L_1^- + L_2^- \ar[d]\ar[r] & 
0\\
0\ar[r] &
{\H}_r^- \ar[d] \ar[r]^(0.3){S_r} &
\ker_{ex}\hat{\mathcal{T}}^*_{\hat{\mathsf{C}}_1}\oplus \ker_{ex}\hat{\mathcal{T}}^*_{\hat{\mathsf{C}}_2} \ar[d] \ar[r]^(0.64){\Delta_+^c} &
\hat{L}_1^- + \hat{L}_2^- \ar[d]\ar[r] &
 0 \\
0\ar[r] &
\ker\Delta_-^0 \ar[d] \ar[r]^{S_r} &
\mathfrak{C}_1^-\oplus \mathfrak{C}_2^- \ar[d] \ar[r]^{\Delta_-^0} &
\mathfrak{C}_1^- + \mathfrak{C}_2^- \ar[d]\ar[r] & 
0 \\
&
0 &
0& 
0& 
\\
}
\tag{\textbf{O}}
\end{equation}
where 
\[
L_i^+ := \partial^c_\infty \ker_{ex}\hat{\mathcal{T}}_{\hat{C}_i}\subset T_{C_\infty}\M_\sigma
\]
\[
 \mathfrak{C}_i^+ := \partial^0_\infty \ker_{ex}\hat{\mathcal{T}}_{\hat{C}_i}\subset T_1G_\infty
\]
\[
L_i^- := \partial^c_\infty \ker_{ex}\hat{\mathcal{T}}^*_{\hat{C}_i}\subset T_{C_\infty}\M_\sigma
\]
\[
 \mathfrak{C}_i^- := \partial^0_\infty \ker_{ex}\hat{\mathcal{T}}^*_{\hat{C}_i}\subset T_1G_\infty
\]

Here is a short explanation of the middle column of the diagram \ref{equ:3T}: We can first look at the beginning of the long exact sequence \ref{equ:longExact}:
\[
\cdots
\to H^0_{\hat{\mathsf{C}}_i} = T_1G_i
\stackrel{\partial_\infty}{\to} H^0(B_{\hat{\mathsf{C}}_0}) = T_1G_\infty
\stackrel{\delta}{\to} H^1(F_{\hat{\mathsf{C}}_0}) = \ker_{\mu}\hat{\mathcal{T}}_{\hat{\mathsf{C}}_i}
\stackrel{\phi}{\to} H^1_{\hat{\mathsf{C}}_i}
\to H^1(B_{\hat{\mathsf{C}}_0})
\to \cdots
\]
Consider $\ker_{ex}\hat{\mathcal{T}}_{\hat{\mathsf{C}}_i} \supset \ker_{\mu}\hat{\mathcal{T}}_{\hat{\mathsf{C}}_i}$. 
Intuitively, $\ker_{\mu}\hat{\mathcal{T}}_{\hat{\mathsf{C}}_i}$ is the tangent space of ``monopoles in $L_{\mu}$ modulo the action of the gauge group in $L_\mu$'', $\ker_{ex}\hat{\mathcal{T}}_{\hat{\mathsf{C}}_i} $ is the tangent space of ``monopoles in $L_{ex}$ modulo the action of the gauge group in $L_\mu$'', and $H^1_{\hat{\mathsf{C}}_i}$ is the tangent space of ``monopoles in $L_{ex}$ modulo the action of the gauge group in $L_{ex}$''. 
Thus the map from $\ker_{ex}\hat{\mathcal{T}}_{\hat{\mathsf{C}}_i}$ to $H^1_{\hat{\mathsf{C}}_i}$ is surjective with the same kernel as $\ker \phi = T_1(G_\infty/\partial_\infty \hat{G}_i)$ (see Lemma 4.3.25 of Nicolaescu's book \cite{Nicolaescu2000NotesOS} for details), and this kernel is $ \mathfrak{C}_i^+$ (see the proof of Propsition \ref{prop:trivialObs}).

\begin{remark}\label{rem:imL(kerLinfty)}
$\delta$ is nontrivial if and only if $\hat{\mathsf{C}}_i$ is irreducible and ${\mathsf{C}}_\infty$ is reducible. We assume this is the case. Then $\ker \phi =T_1(G_\infty/\partial_\infty \hat{G}_0) = \R$ is generated by constant function $\mathbf{i}f \in  T_1G_\infty$.

Now consider the definition of the connecting homomorphism $\delta$. We can choose the preimage of $\mathbf{i}f$ in $T_1 \hat{\mathcal{G}}_{\mu,ex} $ to be the constant function $\mathbf{i}\hat{f}$, or we can choose the preimage to be $ \mathbf{i} \beta(t-r)\hat f$. 
In first case, it's sent to $(\mathbf{i}\hat{f}\hat{\psi},0) \in T_{\hat{\mathsf{C}}_0} \partial_\infty^{-1}(\mathsf{C}_\infty) $, while in the second case, it's sent to $(\mathbf{i}\beta(t-r)\hat{f}\hat{\psi},2\mathbf{i} g dt)$, where $gdt =d (\beta(t-r)\hat{f})$ is a bump function aroud $t=r$. 
These two certainly represent the same class in $H^1(F)$, but only the first one is harmonic and hence in $\ker_{\mu}\hat{\mathcal{T}}_{\hat{\mathsf{C}}_i}$ (By (4.2.2) and Example 4.1.24 of Nicolaescu's book \cite{Nicolaescu2000NotesOS}, elements in $\ker_{\mu}\hat{\mathcal{T}}_{\hat{\mathsf{C}}_i}$ must be harmonic without any $dt$-terms). 
However, the second one, $(\mathbf{i}\beta(t-r)\hat{f}\hat{\psi},2\mathbf{i} g dt)$, shows explicitly that the map $\partial^0_\infty$ in \ref{equ:partial0} is the inverse of $\delta$.
\end{remark}

Here is a short explanation of the middle column of the diagram
\ref{equ:3O}:
$H^2(F_{\hat{\mathsf{C}}_i}) = \ker_{\mu}\hat{\mathcal{T}}^*_{\hat{\mathsf{C}}_i}$ since every self dual $2$-form on $\hat{N}_i$ is in $L_\mu$. On the other hand, the kernel of $\mathfrak{L}_{\hat{\mathsf{C}}_i} $ is exactly $T_1G_i$ which is not in $L_\mu$ (they are constant functions). Hence 
\[
\ker_{ex}\hat{\mathcal{T}}^*_{\hat{C}_i}= \ker_{ex} (\widehat{\underline{SW}}^*_{\hat{\mathsf{C}}_i} \oplus \frac{1}{2}\mathfrak{L}_{\hat{\mathsf{C}}_i})\]
decomposes to the direct sum of $H^2(F_{\hat{\mathsf{C}}_i}) $ and $\mathfrak{C}_i^- = T_1G_i$.

The virtual tangent space and obstruction space will give all monopoles of $\hat{N}(r)$ in a small neighborhood of $\hat{\mathsf{C}}_r$ in its slice:
\begin{theorem}[\cite{Nicolaescu2000NotesOS} Theorem 4.5.7]\label{thm:lgt}
For large enough $r$, the set
\[
\{\hat{\mathsf{C}}; \hat{\mathsf{C}} \text{ are monopoles on } \hat{N}(r), \mathfrak{L}^{*}_{\hat{\mathsf{C}}_r} (\hat{\mathsf{C}} - \hat{\mathsf{C}}_r) = 0, \|\hat{\mathsf{C}} - \hat{\mathsf{C}}_r\|_{2,2} \leq r^{-3}\}
\]
is in one-to-one correspondence with the set 
\[
\{
\hat{\mathsf{C}}_r + \underline{\hat{\mathsf{C}}}_0\oplus \underline{\hat{\mathsf{C}}}^\bot; \|\underline{\hat{\mathsf{C}}}_0\|_{2,2}\leq r^{-3}, \kappa_r (\underline{\hat{\mathsf{C}}}_0) = 0, \underline{\hat{\mathsf{C}}}^\bot = \Phi(\underline{\hat{\mathsf{C}}}_0)
\}
\]
where
\begin{align*}
\hat{\mathsf{C}}_r  &= \hat{\mathsf{C}}_1 \#_r \hat{\mathsf{C}}_2\\
\underline{\hat{\mathsf{C}}}_0 &\in \H_r^+ \\
\underline{\hat{\mathsf{C}}}^\bot &\in \mathcal{Y}_r^-\\
\kappa_r : B_0(r^{-3})\subset \H_r^+ &\to \H_r^- \\
\Phi :B_0(r^{-3})\subset \H_r^+ &\to \mathcal{Y}_r^-
\end{align*}
\end{theorem}

We can also prove that, in the slice of $\hat{\mathsf{C}}_r$, any pair of configurations in small enough neighborhood of $\hat{\mathsf{C}}_r$, are gauge inequivalent (see Lemma 4.5.9 of Nicolaescu's book \cite{Nicolaescu2000NotesOS}). Thus we have
\begin{theorem}[\cite{Nicolaescu2000NotesOS} Corollary 4.5.10]\label{thm:lgc}
For large enough $r$,
\[
\{
\hat{\mathsf{C}}_r + \underline{\hat{\mathsf{C}}}_0\oplus \underline{\hat{\mathsf{C}}}^\bot; \|\underline{\hat{\mathsf{C}}}_0\|_{2,2}\leq r^{-3}, \kappa_r (\underline{\hat{\mathsf{C}}}_0) = 0, \underline{\hat{\mathsf{C}}}^\bot = \Phi(\underline{\hat{\mathsf{C}}}_0), \mathfrak{L}^{*}_{\hat{\mathsf{C}}_r} (\underline{\hat{\mathsf{C}}}_0\oplus \underline{\hat{\mathsf{C}}}^\bot) = 0
\}
\]
is an open set of the moduli space $\M(\hat{N}_r, \hat{\sigma}_1\#\hat{\sigma}_2)$.
\end{theorem}

Moreover, this collection of open sets is an open cover of moduli space $\M(\hat{N}_r, \hat{\sigma}_1\#\hat{\sigma}_2)$:
\begin{theorem}[\cite{Nicolaescu2000NotesOS} Theorem 4.5.15]\label{thm:ls}
Let 
\[
\hat{\mathcal{Z}}_{\Delta} := \{ (\hat{\mathsf{C}}_1,\hat{\mathsf{C}}_2)\in \hat{\mathcal{Z}}_1\times\hat{\mathcal{Z}}_2; \partial_\infty \hat{\mathsf{C}}_1 =  \partial_\infty \hat{\mathsf{C}}_2\}
\]
be the space of compatible monopoles. Then
\[
\bigcup_{\mathsf{C}_r = \hat{\mathsf{C}}_1 \#_r \hat{\mathsf{C}}_2,(\hat{\mathsf{C}}_1,\hat{\mathsf{C}}_2)\in\hat{\mathcal{Z}}_{\Delta}} \{
\hat{\mathsf{C}}_r + \underline{\hat{\mathsf{C}}}_0\oplus \underline{\hat{\mathsf{C}}}^\bot; \|\underline{\hat{\mathsf{C}}}_0\|_{2,2}\leq r^{-3}, \kappa_r (\underline{\hat{\mathsf{C}}}_0) = 0, \underline{\hat{\mathsf{C}}}^\bot = \Phi(\underline{\hat{\mathsf{C}}}_0), \mathfrak{L}^{*}_{\hat{\mathsf{C}}_r} (\underline{\hat{\mathsf{C}}}_0\oplus \underline{\hat{\mathsf{C}}}^\bot) = 0
\}
\]
is $ \M(\hat{N}_r, \hat{\sigma}_1\#\hat{\sigma}_2)$.
\end{theorem}

\subsection{Computation of virtual tangent space and obstruction space}
Now we have stated all results we need. Next we compute the dimension of the moduli space $\dim H^1_{\hat{\mathsf{C}}_0}$ and the dimension of the obstruction space $\dim H^2(F_{\hat{\mathsf{C}}_0})$ for any monopole $\hat{\mathsf{C}}_0$ on $X_0$, $D^3\times \S^1$, and $\S^2\times D^2$.

\begin{prop}\label{prop:dimOfReducible}
Let matrics $g_{bullet}$ be the ones chosen in subsection \ref{subsection:PSCmetric}. Let $\ss(\S^1 \times D^3)$ be the unique $\text{spin}^c$ structure of $\S^1 \times D^3$, and $\ss(D^2 \times \S^2)$ be the unique $\text{spin}^c$ structure of $ D^2 \times \S^2 $ such that the first Chern class of the determinant line bundle is zero. 
Then the moduli space of SW equations without perturbation $\M(\S^1 \times D^3,g_{bullet} ,\ss(\S^1 \times D^3))$ is a circle and $\M(D^2 \times \S^2,g_{bullet} ,\ss(D^2 \times \S^2))$ is a point. 
\end{prop}
\begin{proof}
By the Weitzenb{\"o}ck formula, a non-negative scalar curvature on $3$- or $4$-manifolds leads solely to reducible solutions of the Seiberg-Witten equations (see \cite{kronheimer_mrowka_2007} (4.22)). Hence all monopoles are of the form $(A,0)$, and the Seiberg-Witten equations degenerate to one equation
\[
F_{A}^+ = 0.
\]
Since $F_{A}^+ = \frac{1}{2}(dA + *dA)$ and $\im d\cap \im d^* = \im d\cap \im *d = 0$, $F_{A}^+ = 0$ is equivalent to $dA = 0$.

Fix any $U(1)$-connection $A_0$ of the determinant line bundle of the chosen $\text{spin}^c$ structure. 
In Proposition \ref{prop:spincRestriction} we showed that the first Chern class of the determinant line bundle is zero. Hence $F_{A_0}$ is exact. Let $da_0 = -F_{A_0}$. Then $(A,0)$ is a monopole iff 
\[
A =  A_0 + a_0 + a
\]
for some closed imaginary $1$-form $a$. Hence the space of monopoles is the coset of the space of closed forms.

Now consider the action by the gauge group $\mathscr{G}=Map(M,\S^1)$. Elements in the identity component $I$ of $\mathscr{G}$ can be written as $\e^{\mathbf{i}f}$ where $f$ can be any smooth function ($0$-form), and it changes $A$ by the addition of $\mathbf{i}df$. Also $\mathscr{G}/I = H^1(M;\Z)$. Hence for $M =D^3\times \S^1$ or $\S^2\times D^2$, the moduli space of monopoles can be identified with the torus $H^1(M;\R)/H^1(M;\Z)$.
\end{proof}

By Proposition \ref{prop:spincRestriction} and Proposition \ref{prop:dimOfReducible}, we have 
\begin{corollary}\label{cor:dimOfReducible}
Let $\ss$ be any $\text{spin}^c$ structure of $X$ and $\ss'$ be its unique extension to $X'$ as in Theorem \ref{thm:changeOfSpinc}.  Let matrics $g_{bullet}$ be the ones chosen in subsection \ref{subsection:PSCmetric}. Then the moduli space of SW equations without perturbation $\M(D^3\times \S^1,g_{bullet} ,\ss|_{D^3\times \S^1})$ is a circle and $\M(\S^2\times D^2,g_{bullet} ,\ss'|_{\S^2\times D^2})$ is a point. All monopoles are reducible.
\end{corollary}

\begin{prop}\label{prop:virtualDimX0}
Let $g(X)$ be a metric of $X$ such that $g|_{\partial X_0}$ is the product of canonical metrics on $\S^1$ and $\S^2$. Let $\ss$ be any $\text{spin}^c$ structure of $X$ satisfying the dimension assumption (\ref{equ:dimAssumption}). Let $\hat{\ss}$ be the restriciton of $\ss$ on $X_0$. Then the virtual dimension 
\[
d(\hat{\mathsf{C}}_0)= 1
\]
for any monopole $\hat{\mathsf{C}}_0$ on $X_0$.
\end{prop}
\begin{proof}
Let $\hat{N}$ be a cylindrical manifold with boundary $N = \partial_\infty \hat{N}$. Let $\hat{g}$ be a metric on $\hat{N}$ and $\hat{A}_0$ be a connection on $\hat{N}$. Let $A_0 = \partial_\infty\hat{A}_0$ and $g =\partial_\infty \hat{g}$. Define
\[
\mathbf{F}(g,A_0) := 4\eta_{Dir}({A_0})+ \eta_{sign}(g),
\]
where $\eta_{Dir}({A_0})$ is the eta invariant of the Dirac operator $\mathfrak{D}_{A_0}$, and $\eta_{sign}(g)$ is the eta invariant of the metric $g = \partial_\infty \hat{g}$.

Let ${\mathsf{C}}_\infty = \partial_\infty \hat{\mathsf{C}}_0$. Recall that we always assume that $\hat{\mathsf{C}}_0 \in \hat{\mathcal{C}}_{\mu ,sw}$. Hence ${\mathsf{C}}_\infty$ is a monopole on $\hat{N}$. By Corollary \ref{cor:dimOfReducible}, ${\mathsf{C}}_\infty$ is reducible. 
Then the formula of virtual dimension for the cylindrical manifold $\hat{N}$ is (see page 393 of Nicolaescu's book \cite{Nicolaescu2000NotesOS})
\[
d(\hat{\mathsf{C}}_0)= \frac{1}{4}\left(\int_{\hat{N}} c_1(\hat{A}_0)^2 - 2(\chi_{\hat{N}}+3\sigma_{\hat{N}})\right) + \beta({\mathsf{C}}_\infty),
\]
where
\[
\beta({\mathsf{C}}_\infty) := \frac{1}{2}(b_1(N) - 1) - \frac{1}{4}\mathbf{F}({\mathsf{C}}_\infty).
\]
The integral term is the same as the compact case, and the second term $ \beta({\mathsf{C}}_\infty)$ is called boundary correction term. In our case $N = \partial_\infty\hat{N} =\S^1\times \S^2$, and the metric $\hat{g} = g(X)|_{X_0}$ ensures that $g = \partial_\infty \hat{g}$ is the product of canonical metrics on $\S^1$ and $\S^2$. 
In this situation $\eta_{sign}(g) = 0$ (\cite{Komuro84}) and $\eta_{Dir}(\partial_\infty\hat{C}_0) = 0$ (\cite{Nicolaescu98} Appendix C). Hence $\mathbf{F}(\partial_\infty\hat{C}_0) = 0$. Moreover $b_1(\S^1\times \S^2) = 1$, so $\beta({\mathsf{C}}_\infty) = 0$.

Let $\L$ be the determinant line bundle of $\ss$ and $\hat{\L}$ be the determinant line bundle of $\hat{\ss}$. In the proof of Theorem \ref{thm:changeOfSpinc}, we see that 
\[
c_1(\hat{A}_0)^2 = \langle c_1(\hat{\L})^2, X_0 \rangle = \langle c_1({\L})^2, X \rangle = c_1({\L})^2.
\]
From the triangulation of the boundary sum one can compute that 
\begin{align*}
\chi(X) &=\chi(X_0)+ \chi(\S^1\times D^3) - \chi(\S^1\times \S^2)\\
&=  \chi(X_0)+ (1-1)- (1-1+1-1)\\
&=  \chi(X_0).
\end{align*}
To compute $\sigma(X_0)$ consider the following Mayer-Vietoris sequence
\begin{center}
\begin{tikzpicture}[commutative diagrams/every diagram]
\node (P0) at (0cm, 0cm) {$H^1(X_0)\oplus H^1(\S^1\times D^3)$};
\node (P1) at (-1cm, -0.1cm) {};
\node (P2) at (3.3cm, 0cm) {$H^1(\S^1 \times \S^2)$} ;
\node (P6) at (3.3cm, -1cm) {$\Z$} ;
\node (P3) at (5.5cm, 0cm) {$H^2(X)$};
\node (P4) at (-1cm, -1cm) {$\Z$};
\node (P5) at (8.5cm, 0cm) {$H^2(X_0)\oplus H^2(\S^1\times D^3)$};
\node (P7) at (7.2cm, -0.1cm) {};
\node (P8) at (7.2cm, -1cm) {?};
\node (P9) at (12cm, 0cm) {$H^2(\S^1 \times \S^2)$};
\node (P10) at (12cm, -1cm) {$\Z$};
\path[commutative diagrams/.cd, every arrow, every label]
(P0) edge node {} (P2)
(P2) edge node {$0$} (P3)
(P1) edge node {$\cong$} (P4)
(P3) edge node {$i^*$} (P5)
(P2) edge node {$\cong$} (P6)
(P4) edge node {$\cong$} (P6)
(P7) edge node {$\cong$} (P8)
(P5) edge node {} (P9)
(P8) edge node {$ i_\partial^*$} (P10)
(P9) edge node {$\cong$} (P10);
\end{tikzpicture}
\end{center}
From the assumption of the loop $\gamma$ we choose to do the surgery (the pairing of $\gamma$ and the generator of $H^1(X)=\Z$ is $1$), the dual of $\gamma$ is a $3$-manifold $M\subset X$ and $M\setminus (\S^1\times D^3) \subset X_0$ has the boundary $\{*\}\times \S^2 \subset \S^1\times \S^2 = \partial X_0$. Hence $ i_\partial^* = 0$ and therefore $i^*:H^2(X)\to H^2(X_0)$ is an isomorphism. 
For $2$-manifolds $\Sigma_1,\Sigma_2 \subset X$, we can assume $\gamma \cap \Sigma_i  = \emptyset$ for dimension reason. By choosing a small enough neighborhood of $\gamma$ we can further assume $\Sigma_i \subset X_0$. Hence the pairing of $\Sigma_1$ and $\Sigma_2$ is the same in $X$ and $X_0$. Therefore
\[
\sigma(X_0)= \sigma(X).
\]
Hence $d(\hat{\mathsf{C}}_0)= 1$.
\end{proof}

It turns out that our cases are simple: the obstruction space is trivial.
\begin{prop}\label{prop:noncompactTransversality-general}
Let $\hat{N} = X_0$ and $N= \partial X_0 = \S^1\times \S^2$. Let $\ss$ be any $\text{spin}^c$ structure of $X$. Let $\hat{\ss}$ be the restriciton of $\ss$ on $X_0$. Fix a $\mathsf{C}_\infty \in \M_s$. We can choose a generic perturbation $\eta$ on $X_0$ such that if $\hat{\mathsf{C}}_0$ is an $\eta$-monopole and $\partial_\infty \hat{\mathsf{C}}_0 = \mathsf{C}_\infty$, then $\hat{\mathsf{C}}_0$ is irreducible and $H^2(F(\hat{\mathsf{C}}_0))=0$.
\end{prop}
\begin{proof}
To mimic the definition of the wall in the compact case, define
\begin{equation*}\label{equ:noncampactWall}
\mathcal{W}^{k-1}_{\mu} := \{\eta\in L_\mu^{k-1,2}(i\Lambda^+(X_0));  \exists A\in \mathscr{A}(s),  F^{+_g}_A + i\eta = 0\}.
\end{equation*}
By the computation of the \textbf{ASD} operator $d^+ \oplus d^*
$, one can show that $\mathcal{W}^{k-1}_{\mu}$ is an affine space of codimension $b^+$ (see \cite{Nicolaescu2000NotesOS} Page 404) just as in the compact case. For each $\eta$ outside $\mathcal{W}^{k-1}_{\mu}$, all $\eta$-monopoles are irreducible. Consider the configuration space
\[
\hat{\mathcal{C}}^*_{\mu,sw}/\hat{\mathcal{G}}_{\mu,ex}.
\]
Here $\hat{\mathcal{C}}_{\mu,sw}$ is the space of configurations on $X_0$ that restrict to monopoles on $\partial X_0 = \S^1\times \S^2$, as defined in (\ref{equ:defCsw}).
Let $s=\hat{\ss}|_{\partial X_0}$ and 
\[
\M_s = \M(\S^1\times \S^2,s, g_{round}).
\]
Exactly as in the proof of Proposition \ref{prop:dimOfReducible}, one can show that $\M_s = \S^1$. Let
\[
\mathcal{Z} := \mathcal{Z}^{k-1}_{\mu} := L_\mu^{k-1,2}(i\Lambda^+(X_0)) \setminus \mathcal{W}^{k-1}_{\mu}
\]
be the space of nice perturbations. Consider
\begin{align*}
\mathcal{F} : \hat{\mathcal{C}}^*_{\mu,sw}/\hat{\mathcal{G}}_{\mu,ex}\times \M_s\times \mathcal{Z} &\to \hat{\mathcal{Y}}_\mu \times \M_s\times \M_s\\
( \hat{\mathsf{C}},\mathsf{C},\eta) &\mapsto ( \widehat{{SW}}_{\eta}(\hat{\mathsf{C}}),\partial_\infty \hat{\mathsf{C}}, \mathsf{C}).
\end{align*}
Let $\Delta$ be the diagonal of $\M_s\times \M_s$. One can show that $\mathcal{F}$ is transversal to $0\times \Delta \subset  \hat{\mathcal{Y}}_\mu \times \M_s\times \M_s$ by the diffenrential
\begin{align*}
D_{( \hat{\mathsf{C}}_0, \mathsf{C}_\infty, \eta)}\mathcal{F}: 
 T_{\hat{\mathsf{C}}_0} \mathcal{B}^*_{\mu,sw}\oplus T_{\mathsf{C}_\infty}\M_s\oplus T_\eta \mathcal{Z} &\to T_0\hat{\mathcal{Y}}_{g(b),\mu}\oplus T_{\mathsf{C}_\infty}\M_s \oplus T_{\mathsf{C}_\infty}\M_s\\
(\underline{\hat{\mathsf{C}}}_0,\underline{{\mathsf{C}}}_\infty,\zeta) &\mapsto  (\widehat{\underline{SW}}_{\eta}(\underline{\hat{\mathsf{C}}}_0) +\zeta, \partial_\infty \underline{\hat{\mathsf{C}}}_0,\underline{{\mathsf{C}}}_\infty).
\end{align*} 
Then apply Sard-Smale to the projection
\[
\pi: \mathcal{F}^{-1}(0\times \Delta) \to \mathcal{Z}
\]
to show that $\mathcal{Z}^0_{reg}$, the set of regular values of $\pi$, is of the second category in the sense of Baire (a countable intersection of open dense sets). 

For each $\eta \in \mathcal{Z}^0_{reg}$, the map 
\begin{align*}
\mathcal{F}_\eta : \hat{\mathcal{C}}^*_{\mu,sw}/\hat{\mathcal{G}}_{\mu,ex}\times \M_s&\to \hat{\mathcal{Y}}_\mu \times \M_s\times \M_s\\
( \hat{\mathsf{C}},\mathsf{C}) &\mapsto ( \widehat{{SW}}_{\eta}(\hat{\mathsf{C}}),\partial_\infty \hat{\mathsf{C}}, \mathsf{C}).
\end{align*}
is transversal to $0\times \Delta \subset  \hat{\mathcal{Y}}_\mu \times \M_s\times \M_s$. Let $pr_1$ be the projection to the first summand:
\[
pr_1: \hat{\mathcal{Y}}_\mu \times \M_s\times \M_s \to  \hat{\mathcal{Y}}_\mu.
\]
Then $Dpr_1\circ D\mathcal{F}_\eta$ must be surjective since  $Dpr_1(0\times \Delta)$ is zero. Hence 
\begin{align*}
D_{( \hat{\mathsf{C}}_0, \mathsf{C}_\infty)}(pr_1 \circ \mathcal{F}_\eta): 
 T_{\hat{\mathsf{C}}_0} \mathcal{B}^*_{\mu,sw}\oplus T_{\mathsf{C}_\infty}\M_s &\to T_0\hat{\mathcal{Y}}_{g(b),\mu}\\
(\underline{\hat{\mathsf{C}}}_0,\underline{{\mathsf{C}}}_\infty) &\mapsto  (\widehat{\underline{SW}}_{\eta}\underline{\hat{\mathsf{C}}}_0) .
\end{align*} 
is surjective. This means that $ \widehat{\underline{SW}}_{\eta}$ is surjective, i.e. $H^2_{\hat{\mathsf{C}}_0}= 0$. By the last several terms of the long exact sequence \ref{equ:longExact}
\begin{equation}
\cdots
\to H^1_{\hat{\mathsf{C}}_0}
\stackrel{\partial_\infty}{\to} H^1(B_{\hat{\mathsf{C}}_0})
\to H^2(F_{\hat{\mathsf{C}}_0}) 
\to H^2_{\hat{\mathsf{C}}_0}=0
\to H^2(B_{\hat{\mathsf{C}}_0}) = 0
\to 0,
\tag{\textbf L}
\end{equation}
$H^2(F_{\hat{\mathsf{C}}_0}) =0$ if and only if $\partial_\infty$ is surjective. This is equivalent to say that $\partial_\infty: \widehat{\M}(X_0,\eta) \to \M_s$ is a submersion at $\hat{\mathsf{C}}_0$.

Recall that
\begin{align}
\mathcal{F}_0\begin{pmatrix}
        A\\
          \Phi\\
    \end{pmatrix}
    &=
    \begin{pmatrix}
         d^* A\\
          \D_A \Phi \\
    \end{pmatrix}, \\
\label{equ:0defF1}\mathcal{F}_{1,\eta}\begin{pmatrix}
        A\\
          \Phi\\
    \end{pmatrix}
    &=     F^{+}_A + i\eta - \rho^{-1}(\sigma(\Phi, \Phi)).
\end{align}
Fix a $\mathsf{C}_\infty \in \M_s$, then
\begin{align*}
\mathcal{F}_{\mathsf{C}_\infty }:  \partial_\infty^{-1}(\mathsf{C}_\infty)/\hat{\mathcal{G}}_{\mu}  \times \mathcal{Z} &\to \hat{\mathcal{Y}}_\mu\\
( \hat{\mathsf{C}},\eta) &\mapsto  \mathcal{F}_{1,\eta}(\hat{\mathsf{C}})
\end{align*}
is transversal to $0 \in  \hat{\mathcal{Y}}_\mu $. As above, we can find a set $\mathcal{Z}_{reg}^1$ of the second category in the sense of Baire, such that for each $\eta \in \mathcal{Z}_{reg}$, $\mathcal{F}_{\mathsf{C}_\infty,\eta } =\mathcal{F}_{1,\eta}$ is transversal to $0 \in  \hat{\mathcal{Y}}_\mu $. This means that
\begin{equation}\label{equ:surjective-for-solution}
 H^2(F_{\hat{\mathsf{C}}_0}) =0
\end{equation} 
for any $\hat{\mathsf{C}}_0 \in ( \partial_\infty^{-1}(\mathsf{C}_\infty)/\hat{\mathcal{G}}_{\mu}  )\cap \mathcal{F}_{1,\eta}^{-1}(0)$. 
\end{proof}

\begin{prop}\label{prop:noncompactTransversality}
Let $\hat{N} = X_0$ and $N= \partial X_0 = \S^1\times \S^2$. Suppose $\S^1\times \{pt\} \subset \partial X_0 \subset X_0$ represents a generator of $H_1(X_0;\Z) =\Z$. Let $\ss$ be any $\text{spin}^c$ structure of $X$. Let $\hat{\ss}$ be the restriciton of $\ss$ on $X_0$. We can choose a generic perturbation $\eta$ on $X_0$ such that if $\hat{\mathsf{C}}_0$ is an $\eta$-monopole, it is irreducible and $H^2(F(\hat{\mathsf{C}}_0))=0$.
\end{prop}
\begin{proof}

We just need to show that the perturbation $\eta$ in Propsition \ref{prop:noncompactTransversality-general} works for all $\partial_\infty \hat{\mathsf{C}}_0 \in \M_s$. Pick any $\mathsf{C}_\infty \in \M_s$, remain to show that if $\hat{\mathsf{C}}_0$ is an $\eta$-monopole and $\partial_\infty \hat{\mathsf{C}}_0 = \mathsf{C}_\infty$, $\hat{\mathsf{C}}_0$ is irreducible and $H^2(F(\hat{\mathsf{C}}_0))=0$.
 
Let $(0, A)$ be a representative of $ {\mathsf{C}}_\infty$. Choose any $(\hat\Phi, \hat A)\in \hat{\mathcal{C}}^*_{\mu,sw}$, then $\partial_\infty (\hat\Phi, \hat A)$ is an $(\eta|_N)$-monopole on ${N}$. We want to show that even if $\partial_\infty(\hat\Phi, \hat A)$ does not represent $ {\mathsf{C}}_\infty$, $d_{(\hat\Phi, \hat A)}\mathcal{F}_1|_{\partial_\infty^{-1}(\partial_\infty(\hat A))/\hat{\mathcal{G}}_{\mu}} $ is still surjective.

Since $\eta$ is zero on the neck, $\partial_\infty \hat A$ is closed and $\partial_\infty \hat\Phi =0$  (see the proof of Proposition \ref{prop:dimOfReducible}). 
Hence $ \partial_\infty(\hat A)-A$ is closed. Since $H^1(\hat{N};\R)\to H^1(N;\R)$ is surjective, 
one can find a closed form $\hat a$ on $\hat N$ such that $ \partial_\infty(\hat A)-A =\partial_\infty(\hat a) +df $ for some function $f$ on $N$. Hence $ \partial_\infty(\hat A + \hat a) = A +df$, which belongs to the gauge equivalence class of $A$. This means $\partial_\infty(\hat A + \hat a) = {\mathsf{C}}_\infty$. 
Because $\hat a$ is closed, if $(\hat\Phi, \hat A)$ is a solution of $\mathcal{F}_{1,\eta}$, $(\hat\Phi, \hat A + \hat a)$ is also a solution of $\mathcal{F}_{1,\eta}$. By (\ref{equ:surjective-for-solution}),  
\begin{align}
 d_{(\hat\Phi, \hat A + \hat a)}(\mathcal{F}_1|_{\partial_\infty^{-1}(\mathsf{C}_\infty)/\hat{\mathcal{G}}_{\mu}}): T_{(\hat\Phi, \hat A + \hat a)} \partial_\infty^{-1}(\mathsf{C}_\infty)/\hat{\mathcal{G}}_{\mu} &\to \hat{\mathcal{Y}}_\mu\\
 (\alpha,\phi) &\mapsto d^+ \alpha  -\rho^{-1}(\sigma(\hat\Phi, \phi)+\sigma(\phi, \hat\Phi)) \label{equ:differentialOfF1}
\end{align}
is surjective. Note that $d_{(\hat\Phi, \hat A)}\mathcal{F}_1$ does not depend on $\hat A$. Also an element of either $T_{(\hat\Phi, \hat A )} \partial_\infty^{-1}(\partial_\infty(\hat A))/\hat{\mathcal{G}}_{\mu}$ or $T_{(\hat\Phi, \hat A + \hat a)} \partial_\infty^{-1}(\mathsf{C}_\infty)/\hat{\mathcal{G}}_{\mu}$ can be written as $(\alpha,\phi)$ such that $\partial_\infty \alpha $ represents $0\in H^1(N;\R)$.
Hence
\[
d_{(\hat\Phi, \hat A)}\mathcal{F}_1|_{\partial_\infty^{-1}(\partial_\infty(\hat A))/\hat{\mathcal{G}}_{\mu}} = d_{(\hat\Phi, \hat A + \hat a)}\mathcal{F}_1|_{\partial_\infty^{-1}(\mathsf{C}_\infty)/\hat{\mathcal{G}}_{\mu}}
\]
is surjective.

Let
\[
\mathcal{Z}_{reg} = \mathcal{Z}^0_{reg} \cap \mathcal{Z}^1_{reg}.
\]
For any $\eta \in \mathcal{Z}_{reg}$, if $\hat{\mathsf{C}}_0$ is an $\eta$-monopole, it is irreducible and $H^2(F(\hat{\mathsf{C}}_0))=0$. Moreover, $\mathcal{Z}_{reg}$ is still a countable intersection of open and dense sets, so it is of the second category in the sense of Baire.
\end{proof}

\begin{remark}
The statement of Proposition \ref{prop:noncompactTransversality} is not true in general. If the boundary $N = \S^1\times \S^2$ and $L^1_{top} =0$, we must have 
\[
\dim H^2(F(\hat{\mathsf{C}}_0)) = \dim H^2( \widehat{\mathcal{K}}_{\hat{\mathsf{C}}_0}) + 1.
\] 
To prove this, it suffices to find an element in $T_{\hat{\mathsf{C}}_0} \hat{\mathcal{C}}_{\mu,sw} $, such that its image is not in the image of $T_{\hat{\mathsf{C}}_0} \partial_\infty^{-1}(\mathsf{C}_\infty) $. Indeed, there exists a $1$-form $\alpha \in \Omega^1(\hat N)$, such that $\partial_\infty \alpha $ generates $H^1(N)$ (namely $(\alpha,0)\notin T_{\hat{\mathsf{C}}_0} \partial_\infty^{-1}(\mathsf{C}_\infty) $), and $d^+ \alpha$ is a nonzero element in $H^2(\hat{N})$. 
Conversely, if $d^+ \alpha'$ is nonzero in $H^2(\hat{N})$, then it's not compactly supported, otherwise it would be orthogonal to any self dual harmornic $2$-forms. Hence $\partial_\infty \alpha' $ is nonzero in $H^1(N)$. Therefore 
\begin{equation}\label{extraVector}
\left. d\mathcal{F}_1 (\alpha,0) \neq d\mathcal{F}_1\right|_{T ( \partial_\infty^{-1}(\mathsf{C}_\infty) )} (\beta,0)
\end{equation}
for any $(\beta,0) \in T ( \partial_\infty^{-1}(\mathsf{C}_\infty) )$. When the virtual dimension of the moduli space is less then $1$, for a generic perturbation such that for any solution $(\hat{A},\hat{\Phi})$,
\begin{equation}\label{extraVector2}
d_{(\hat{A},\hat{\Phi})}\mathcal{F}_1 (\alpha,0) \notin \text{im}\left. d_{(\hat{A},\hat{\Phi})}\mathcal{F}_1 \right|_{T\partial_\infty^{-1}(\mathsf{C}_\infty) },
\end{equation}
even though $\hat{b}_+ > 0$. This is because in this case the connection part is not able to kill $d_{(\hat{A},\hat{\Phi})}\mathcal{F}_1 (\alpha,0) $ by (\ref{extraVector}), and the spinor part is responsible to kill the other complement, instead of $d_{(\hat{A},\hat{\Phi})}\mathcal{F}_1 (\alpha,0) $, otherwise it will produce one more dimension of the cokernel and one more dimension of the moduli space, which would not happen by the classical transversality argument. Hence $\dim H^2(F(\hat{\mathsf{C}}_0)) = \dim H^2( \widehat{\mathcal{K}}_{\hat{\mathsf{C}}_0}) + 1$ for any solution $\hat{\mathsf{C}}_0$.

In fact, the condition on the virtual dimension can be omitted. $d^+ \alpha$ is not compactly supported , and the harmonic projection $\H(d^+ \alpha)$ satisfies
\[
\partial_\infty^0 \H(d^+ \alpha) \neq 0
\]
On the other hand, the second term 
\[
-\rho^{-1}(\sigma(\hat\Phi, \phi)+\sigma(\phi, \hat\Phi))
\]
of (\ref{equ:differentialOfF1}) is in $L_\mu$ since $\partial_\infty \hat{\Phi} = 0$. Hence (\ref{extraVector2}) is true as long as all solutions on the boundary are reducible.

This example is a counter example of \cite{Nicolaescu2000NotesOS} Proposition 4.4.1. The equation
\[
\dim H^2( \widehat{\mathcal{K}}_{\hat{\mathsf{C}}_0})  = \hat{b}_+
\]
for $\hat \Phi=0$ computed in \cite{Nicolaescu2000NotesOS} Page 404, combined with the equation
\[
\left. \dim \ker_{ex} (\textbf{ASD}^*) \right|_{\Omega^2(\hat{N})} =  \hat{b}_+ + \dim L^2_{top}
\]
computed in \cite{Nicolaescu2000NotesOS} Page 312, also shows the the existence of $\alpha$ satisfying (\ref{extraVector}) without any explicit construction.
\end{remark}

\begin{prop}\label{prop:triObs}
For $\hat{N}= X_0$, $\widehat{\S^1\times \S^3}$ or $D^3\times \S^1$ with positive scalar curvature metric $\hat{g}$ chosen in subsection \ref{subsection:PSCmetric}, We can choose suitable perturbations $\eta = \eta(\hat{N})$ such that if $\hat{\mathsf{C}}_0$ is an $\eta$-monopole, $H^2(F(\hat{\mathsf{C}}_0))=0$.
\end{prop}

\begin{proof}
As in the usual argument of transversality, we just need to take care of the boundary term to prove that, if $b_+(\hat{N})>0$, we can choose a pertubation $\eta \in H^2_+(\hat{N})$ such that all $\eta$-monopoles are strongly regular (and irreducible) (Proposition \ref{prop:noncompactTransversality}). Since $H^2_+(X_0)$ is assumed to be nontrivial, the statement is true for $X_0$.

For $\hat{N} \simeq \S^1\times \S^3$, $D^3\times \S^1$ or $\S^2\times D^2$, all monopoles are reducible. Let $\hat{\mathsf{C}}_0 = (\hat{A}_0,0)$ be a reducible monopole for the SW equations without perturbation. The connection $\hat{A}_0$ on the cylindrical manifold $\hat{N}$ gives an asymptotically cylindrical Dirac operator $\D^*_{\hat{A}_0}$ with
\[
\partial_\infty \D^*_{\hat{A}_0} = \mathfrak{D}^*_{{A}_0}.
\]
The middle column of the Obstruction space diagram \ref{equ:3O} comes from the exact sequence (\cite{Nicolaescu2000NotesOS} Proposition 4.3.30)
\[
0\to H^2(F(\hat{\mathsf{C}}_0)) \to \text{ker}_{\text{ex}}\hat{\mathcal{T}}^*_{\hat{C}_0} \stackrel{\partial_\infty^0}{\to} \im(T_1\hat{G}_0 \stackrel{\partial_\infty}{\to} T_1G_\infty) \to 0.
\]
Recall that in (\ref{defOfTWithoutMu}) we define
\[
\hat{\mathcal{T}}_{\hat{\mathsf{C}}_0} := \widehat{\underline{SW}}_{\hat{\mathsf{C}}_0} \oplus \frac{1}{2}\mathfrak{L}^*_{\hat{\mathsf{C}}_0}.
\]
If $\mathbf{i}f\in T_1\hat{G}_0$, then it's in the kernel of $\mathfrak{L}_{\hat{\mathsf{C}}_0}$, and therefore in $\ker_{ex}\hat{\mathcal{T}}^*_{\hat{C}_0}= \ker_{ex} (\widehat{\underline{SW}}^*_{\hat{\mathsf{C}}_0} \oplus \frac{1}{2}\mathfrak{L}_{\hat{\mathsf{C}}_0})$. On the other hand, if
\[
(\Psi, \mathbf{i}f)\in  L_{ex}^{1,2}(\hat{\S}_{\hat{\sigma}}^- \oplus \mathbf{i}\Lambda_+^2T^*\hat{N}) \oplus L_{ex}^{1,2}( \mathbf{i}\Lambda^0 T^*\hat{N})
\]
is in $\ker_{ex}\hat{\mathcal{T}}^*_{\hat{C}_i}$, then $\mathbf{i}f\in T_1G_0$. Thus
\[
  \partial^0_\infty \ker_{ex}\hat{\mathcal{T}}^*_{\hat{C}_0}
 \cong \partial_\infty T_1G_0.
\]
Namely, $H^2(F(\hat{\mathsf{C}}_0))$ doesn't contain constant functions. Hence 
\[H^2(F(\hat{\mathsf{C}}_0)) = \ker_{\text{ex}}\D^*_{\hat{A}_0} \oplus \ker_{\text{ex}}(d^+ \oplus d^*)^*|_{\Lambda_+^2(T^*\hat{N}) \oplus \Lambda^0_0 (T^*\hat{N})}.
\]
Then by the computation of the \textbf{ASD} operator $d^+ \oplus d^*
$ (\cite{Nicolaescu2000NotesOS} Example 4.1.24), 
\begin{equation}\label{equ:obstructionF=kerexD*}
H^2(F(\hat{\mathsf{C}}_0)) = \text{ker}_{\text{ex}}\D^*_{\hat{A}_0}\oplus H^2_+(\hat{N})\oplus L^2_{top}, 
\end{equation}
where $L^2_{top}=\im( i^*:H^2(\hat{N})\to H^2(\partial\hat{N}))$ for inclusion map $i:N\to \hat{N}$. Thus the second and the third components are trivial for $\hat{N}=\widehat{\S^1\times \S^3} $ or ${D^3\times \S^1}$. Now compute the dimension of $\text{ker}_{\text{ex}}\D^*_{\hat{A}_0} $. Since each of them has a positive scalar curvature metric, by the Weitzenb{\"o}ck formula, the twisted Dirac operater is invertible since $A_0$ is flat. This means that $\ker \mathfrak{D}^*_{A_0} = 0$ and therefore
\begin{equation}\label{equ:kerex=ker}
\ker_{ex}\D^*_{\hat{A}_0} = \ker_{L^2}\D^*_{\hat{A}_0}.
\end{equation}
Hence
\[
I_{APS}(\D_{\hat{A}_0} ) = \dim_{\C} \ker_{L^2} \D_{\hat{A}_0} -  \dim_{\C} \ker_{L^2} \D^*_{\hat{A}_0},
\]
where $I_{APS}(\hat{L})$ is the Atiyah-Patodi-Singer index of the APS operator $\hat{L}$. One can also prove that $ \ker_{L^2} \D_{\hat{A}_0}$ is trivial by the Weitzenb{\"o}ck formula (see \cite{Nicolaescu2000NotesOS} Page 323). Hence 
\[
-\dim \text{ker}_{\text{ex}}\D^*_{\hat{A}_0}= I_{APS}(\D_{\hat{A}_0} ).
\]
By the Atiyah-Patodi-Singer index theorem (\cite{atiyah_patodi_singer_1975}) we have 
\[
I_{APS}(\D_{\hat{A}_0} ) = \frac{1}{8} \int_{\hat{N}} (p_1(\hat{\nabla}^{\hat{g}}) + c_1(\hat{A}_0)^2) -\frac{1}{2}(\dim\ker\mathfrak{D}_{A_0}  +\eta_{Dir}({A_0})),
\]
where $\hat{\nabla}^{\hat{g}}$ is the Levi-Civita connection of $\hat{g}$, $p_1(\hat{\nabla}^{\hat{g}})$ and $c_1(\hat{A}_0)$ are the first Pontryagin class and the first Chern class determined by the Chern-Weil construction, and $\eta_{Dir}({A_0})$ is the eta invariant of the Dirac operator $\mathfrak{D}_{A_0}$. 
For any $4$-manifold with boundary, one has ``signature defect'' (see \cite{Nicolaescu2000NotesOS} (4.1.34), see also \cite{atiyah_patodi_singer_1975}, \cite{atiyah_patodi_singer_1975II} and \cite{atiyah_patodi_singer_1976} for the motivation)
\[
\eta_{sign}(g) = \frac{1}{3}\int_{\hat{N}} p_1(\hat{\nabla}^{\hat{g}}) - \sigma(\hat{N}) 
\]
where $\eta_{sign}(g)$ is the eta invariant of the metric $g = \partial_\infty \hat{g}$. Also recall that
\[
\mathbf{F}(g,A_0) := 4\eta_{Dir}({A_0})+ \eta_{sign}(g).
\]
Combine all of these, one has 
\[
8\dim \text{ker}_{\text{ex}}\D^*_{\hat{A}_0} =\mathbf{F}(\partial_\infty\hat{C}_0)  + \sigma(\hat{N})  - \int_{\hat{N}} c_1(\hat{A}_0)^2.
\]
For $\hat{N} \simeq \S^1\times \S^3$, $D^3\times \S^1$ or $\S^2\times D^2$, $\sigma(\hat{N})  = 0$. 
For $\hat{N} \simeq D^3\times \S^1$ or $\S^2\times D^2$, $N = \partial_\infty\hat{N} =\S^1\times \S^2$, and the metric $\hat{g}$ chosen in subsection \ref{subsection:PSCmetric} ensures that $g = \partial_\infty \hat{g}$ is the product of canonical metric on $\S^1$ and $\S^2$. 
In this situation $\eta_{sign}(g) = 0$ (\cite{Komuro84}) and $\eta_{Dir}(\partial_\infty\hat{C}_0) = 0$ (\cite{Nicolaescu98} Appendix C). Hence $\mathbf{F}(\partial_\infty\hat{C}_0) = 0$. 
For $\hat{N} \simeq D^3\times \S^1$ or $\S^2\times D^2$, as shown in Proposition \ref{prop:dimOfReducible}, $\hat{A}_0$ is a flat connection. Hence for $\hat{N} \simeq D^3\times \S^1$ or $\S^2\times D^2$, $\dim \text{ker}_{\text{ex}}\D^*_{\hat{A}_0}=0$. 
So the first component of $H^2(F(\hat{\mathsf{C}}_0)) $ is also trivial. Thus $\hat{C}_0$ is strongly regular for $\hat{N}=\widehat{\S^1\times \S^3} $ or ${D^3\times \S^1}$ without perturbations.
\end{proof}

For $\hat{N}=\S^2\times D^2$, unfortunately, $L^2_{top}$ is $1$-dimensional ($ i^*:H^2({\S^2\times D^2})\to H^2(\S^2\times \S^1)$ is an isomorphism between two copies of $\Z$), so $H^2(F(\hat{\mathsf{C}}_2))$ is $1$-dimensional in the obstruction diagram for $\hat{\mathsf{C}}_1$ on $X_0$ and $\hat{\mathsf{C}}_2$ on $\S^2\times D^2$. However, we have
\begin{prop}\label{prop:trivialObs}
When $r$ is large enough, the obstruction space $\mathcal{H}^-_r$ for $X'=X_0 \cup_{\S^1\times \S^2}D^2\times  \S^2$ is still $0$.
\end{prop}
\begin{proof}
Let $\hat{N}_1 = X_0$, $\hat{N}_2 = \S^2\times D^2$. Then $N=\partial_\infty \hat{N}_i= \S^2\times \S^1$. The method is to trace through the obstruction diagram.

First, by Proposition \ref{prop:triObs}, $H^2(F(\hat{\mathsf{C}}_1))=0$, and $H^2(F(\hat{\mathsf{C}}_2))\cong \R$.

Next, we identify $\mathfrak{C}_i^-$. Recall that
\[
\hat{\mathcal{T}}_{\hat{\mathsf{C}}_i} := \widehat{\underline{SW}}_{\hat{\mathsf{C}}_i} \oplus \frac{1}{2}\mathfrak{L}^*_{\hat{\mathsf{C}}_i}.
\]
If $\mathbf{i}f\in T_1G_i$, then it's in the kernel of $\mathfrak{L}_{\hat{\mathsf{C}}_i}$, and therefore in $\ker_{ex}\hat{\mathcal{T}}^*_{\hat{C}_i}= \ker_{ex} (\widehat{\underline{SW}}^*_{\hat{\mathsf{C}}_i} \oplus \frac{1}{2}\mathfrak{L}_{\hat{\mathsf{C}}_i})$. On the other hand, if
\[
(\Psi, \mathbf{i}f)\in  L_{ex}^{1,2}(\hat{\S}_{\hat{\sigma}}^- \oplus \mathbf{i}\Lambda_+^2T^*\hat{N}_i) \oplus L_{ex}^{1,2}( \mathbf{i}\Lambda^0 T^*\hat{N}_i)
\]
is in $\ker_{ex}\hat{\mathcal{T}}^*_{\hat{C}_i}$, then $\mathbf{i}f\in T_1G_i$. Thus
\[
\mathfrak{C}_i^- =  \partial^0_\infty \ker_{ex}\hat{\mathcal{T}}^*_{\hat{C}_i}
 \cong \partial_\infty T_1G_i.
\]

For manifolds with cylindrical end, we can choose a generic perturbation in a $b^+$-dimensional space just as in the compact case (see page 404 of Nicolaescu's book \cite{Nicolaescu2000NotesOS} for a proof). 
Since $b^+( X_0)>0$, we can choose a compactly supported $2$-form $\eta$ such that all monopoles on $\hat{N}_1 = X_0$ are irreducible. 
Since $\hat{N}_2 = \S^2\times D^2$ and $N= \S^2\times \S^1$ admit PSC metric, all monopoles on $\hat{N}_2 = \S^2\times D^2$ and $N$ are reducible. So $\mathfrak{C}_1^- = 0$ and $ \mathfrak{C}_2^- \cong \R$. So $\Delta_-^0$ is an isomorphism in the obstruction diagram.  
Since each row of the diagram is asymptotically exact, any unit vector of $S_r (\ker\Delta_-^0) $ approaches $0$ as $r\to \infty$. So $S_r (\ker\Delta_-^0)=0$ and thus $\ker\Delta_-^0$ must be trivial when $r$ is large enough. Since each column of the diagram \ref{equ:3O} is exact, ${\H}_r^-  \cong \ker\Delta_-^c$.

Next we identify $L_i^-$. We have assumed $\M(X_0)$ is $1$-dimensional, and since $D^2\times  \S^2$ has a PSC metric and $H^1(D^2\times  \S^2) = 0$, $\M(D^2\times  \S^2)$ is only one reducible point. $\S^1\times \S^2$ also has a PSC metric and $H^1(\S^1\times \S^2) = 0$, so $\M(\S^1\times \S^2)$ is a circle of reducible solutions. So 
\begin{align*}
\dim_\R H^1_{\hat{\mathsf{C}}_1} &= 1, \\
\dim_\R H^1_{\hat{\mathsf{C}}_2} &= 0, \\
\dim_\R T_{C_\infty}\M_\sigma &= 1. 
\end{align*}
In the first row of diagram \ref{equ:3T}, $L^+_2 = \Delta_+^c (H^1_{\hat{\mathsf{C}}_2})$. Hence $L^+_2$ is certainly $0$. By complementarity equations from the Lagrangian condition (see (4.1.22) of Section 4.1.5 of Nicolaescu's book), we have 
\[
L_i^+\oplus L_i^- =  T_{C_\infty}\M_\sigma.
\]
So $L_2^-$ is $\R$. Thus in the first row of obstruction diagram \ref{equ:3O}, $L_1^- + L_2^- = \R$. Since $H^2(F(\hat{\mathsf{C}}_1))\oplus H^2(F(\hat{\mathsf{C}}_2)) = \R$, $\Delta_-^c $ is an isomorphism and ${\H}_r^-  \cong \ker\Delta_-^c=0$.
\end{proof}

\subsection{Global gluing theorem}\label{section:global-gluing}
We already have local gluing results. Now we can combine them to prove that, the moduli space of solutions of the new manifold is the fiber product of two old moduli spaces. 

We assume the following:

$\mathbf{A_1}$  $(N,g)$ is $\S^3$ or $\S^1\times \S^2$ with a positive scalar metric.

$\mathbf{A_2}$  $b_+(\hat{N}_1) > 0$, $b_+(\hat{N}_2)=0$.

$\mathbf{A_3}$  All the finite energy monopoles on $\hat{N}_1$ are irreducible and strongly regular.

$\mathbf{A_4}$  Any finite energy ${\hat{\sigma}_2}$-monople $\hat{\mathsf{C}}_2$ is reducible and $\dim_{\R} H^1_{\hat{\mathsf{C}}_2}$ is $0$ or $1$.

$\mathbf{A_5}$ The obstruction space $\mathcal{H}^-_r$ is $0$ when $r$ is large enough.

Recall that
\[
\hat{\mathcal{Z}}_{\Delta} := \{ (\hat{\mathsf{C}}_1,\hat{\mathsf{C}}_2)\in \hat{\mathcal{Z}}_1\times\hat{\mathcal{Z}}_2; \partial_\infty \hat{\mathsf{C}}_1 =  \partial_\infty \hat{\mathsf{C}}_2\}
\]
is the space of compatible monopoles, and $\widehat{\mathcal{G}}_i$ is the gauge group on $\hat{N}_i$. Define
\[
\widehat{\mathcal{G}}_\Delta := \{(\hat{\gamma}_1,\hat{\gamma}_2) \in \widehat{\mathcal{G}}_1\times\widehat{\mathcal{G}}_2 ; \partial_\infty\hat{\gamma}_1 =  \partial_\infty\hat{\gamma}_2\}.
\]

Let
\[
\hat{\mathfrak{N}}:= \hat{\mathcal{Z}}_{\Delta}  /\widehat{\mathcal{G}}_\Delta .
\]
The cutoff trick described before (see \ref{equ:cut} and \ref{equ:paste}) gives gluing maps
\[
\#_r: \widehat{\mathcal{G}}_\Delta\to \widehat{\mathcal{G}}_{\hat{N}_r}
\]
and
\[
\#_r: \hat{\mathcal{Z}}_{\Delta} \to \mathcal{C}_{\hat{N}_r}.
\]
The second one is  $(\widehat{\mathcal{G}}_\Delta , \widehat{\mathcal{G}}_{\hat{N}_r})$-equivariant, since these gluing maps share the same parameter $r$. So we can mod out by the $(\widehat{\mathcal{G}}_\Delta , \widehat{\mathcal{G}}_{\hat{N}_r})$-action, and get
\[
\hat{\#}_r: \hat{\mathfrak{N}} \to \hat{\mathcal{B}}_{\hat{N}_r}
\]
We also denote the image of this map by $\hat{\mathfrak{N}}$.

\begin{theorem}\label{thm:ggt}
Under assumptions ($\mathbf{A_1}$) - ($\mathbf{A_5}$), for large enough $r$, $\hat{\#}_r\hat{\mathfrak{N}}$ is isotopic to the moduli space of genuine monopoles ${\M}(\hat{N}_r)$ as submanifolds of $\hat{\mathcal{B}}_{\hat{N}_r}$.

\end{theorem}
\begin{proof}
For any point $ (\hat{\mathsf{C}}_1,\hat{\mathsf{C}}_2)$ in $\hat{\mathcal{Z}}_{\Delta}$, let 
\[
\mathsf{C}_r = \#_r  (\hat{\mathsf{C}}_1,\hat{\mathsf{C}}_2) = \hat{\mathsf{C}}_1 \#_r \hat{\mathsf{C}}_2.
\]
By assumption $\mathbf{A_1}$, all monopoles on $N$ are reducible. Thus $T_1G_\infty = \R$. By assumption $\mathbf{A_4}$, $\mathfrak{C}_2^- = \R$, so that $\mathfrak{C}_2^+ = 0$. Hence $\Delta^0_+$ must be an isomorphism in the last row of diagram \ref{equ:3T}. So 
\begin{equation}\label{equ:H+=kerDelta}
\H_r^+ \cong^a \ker\Delta_+^c,
\end{equation}
where $ \cong^a$ means that the isomorphism is given by an asymptotic map in the sense of \cite{Nicolaescu2000NotesOS} page 301.

Now we want to show 
\begin{equation}\label{equ:kerDelta=Tangent}
\ker\Delta_+^c \cong T_{[\mathsf{C}_r]}\hat{\mathfrak{N}}.
\end{equation}
By the definition of $H^1_{\hat{\mathsf{C}}_i}$ and boundary difference map $\Delta_+^c $, a point in $\ker\Delta_+^c $ is a pair $(\underline{\hat{\mathsf C}}_1, \underline{\hat{\mathsf C }}_2)\in \mathcal{S}_{\hat{\mathsf{C}}_1} \times \mathcal{S}_{\hat{\mathsf{C}}_2}$ in the local slice of monopoles, such that $\partial_\infty \underline {\hat{\mathsf{C}}}_1 =  \partial_\infty \underline{\hat{\mathsf{C}}}_2$. 
On the other hand, any point of $T_{[\mathsf{C}_r]}\hat{\mathfrak{N}}$ can be represented by $(\hat{\gamma}_1\underline{\hat{\mathsf C}}_1, \hat{\gamma}_2\underline{\hat{\mathsf C }}_2)\in T\hat{\mathcal{Z}}_{\Delta}$ for  $(\underline{\hat{\mathsf C}}_1, \underline{\hat{\mathsf C }}_2)\in\ker\Delta_+^c$ and $(\hat{\gamma}_1,\hat{\gamma}_2) \in \widehat{\mathcal{G}}_1 \times\widehat{\mathcal{G}}_2$, by the definition of slice. Since $\underline{\hat{\mathsf C}}_1$ and $\underline{\hat{\mathsf C}}_2$ have the same boundary value, and $(\hat{\gamma}_1\underline{\hat{\mathsf C}}_1, \hat{\gamma}_2\underline{\hat{\mathsf C }}_2) \in T\hat{\mathcal{Z}}_{\Delta}$, $\hat{\gamma}_1$ and $\hat{\gamma}_2$ must coincide on the boudary. Thus $(\hat{\gamma}_1,\hat{\gamma}_2) \in T\widehat{\mathcal{G}}_\Delta$. Therefore, $\ker\Delta_+^c \cong T_{[\mathsf{C}_r]}\hat{\mathfrak{N}}$.

By (\ref{equ:H+=kerDelta}) and (\ref{equ:H+=kerDelta}), the family of $\H_r^+$ indexed by $\mathsf{C}_r$ forms the tangent bundle of $\hat{\mathfrak{N}}$ when $r$ is sufficiently large. We again denote it by $\H_r^+$. By the definition of $\mathcal{Y}_r^+$, it's the normal bundle of $\hat{\mathfrak{N}}$ in $\hat{\mathcal{B}}_{\hat{N}_r}$. By condition $\mathbf{A_5}$, the map $\kappa_r$ in theorem \ref{thm:lgc} must be zero. We conclude that ${\M}(\hat{N}_r) $ is a section of the normal bundle of $ \hat{\mathfrak{N}}$ locally. Thus for each $\mathsf{C}_r$, there exists an open neighborhood $U_r$, such that ${\M}(\hat{N}_r) \cap U_r \cong  \hat{\mathfrak{N}}\cap U_r $. By theorem \ref{thm:ls}, this fact is globally true.
\end{proof}

Now we can show that $\hat{\mathfrak{N}}$ above is desired fiber product of moduli space.
\begin{lemma}\label{lem:fibProd}
Let $\mathcal{Z}$ be monopoles on $N$. Define
\[
\mathcal{G}^{\partial_\infty}:= \partial_\infty\widehat{\mathcal{G}}_1 \cdot  \partial_\infty\widehat{\mathcal{G}}_2,
\]
\[
\M^{\partial_\infty} := \mathcal{Z}/\mathcal{G}^{\partial_\infty},
\]
\[
\hat{\mathcal{Z}} := \{ (\hat{\mathsf{C}}_1,\hat{\mathsf{C}}_2)\in \hat{\mathcal{Z}}_1\times\hat{\mathcal{Z}}_2; \partial_\infty \hat{\mathsf{C}}_1 \equiv  \partial_\infty \hat{\mathsf{C}}_2 \mod \mathcal{G}^{\partial_\infty} \}.
\]
Then we have
\[
\hat{\mathcal Z}/\widehat{\mathcal{G}}_1\times\widehat{\mathcal{G}}_2 = 
 \{ ([\hat{\mathsf{C}}_1],[\hat{\mathsf{C}}_2])\in \hat{\M}_1\times\hat{\M}_2; \partial_\infty [\hat{\mathsf{C}}_1] =  \partial_\infty [\hat{\mathsf{C}}_2] \in \M^{\partial_\infty}\}
\]
and
\[
\hat{\mathcal Z}/\widehat{\mathcal{G}}_1\times\widehat{\mathcal{G}}_2 \cong \hat{\mathfrak{N}}.
\]
\end{lemma}
\begin{proof}
The first equility is just by definition. We prove the second one:

$\hat{\mathfrak{N}}$ is certainly a subset of $\hat{\mathcal Z}/\widehat{\mathcal{G}}_1\times\widehat{\mathcal{G}}_2$. For any $([\hat{\mathsf{C}}_1],[\hat{\mathsf{C}}_2])$ in $\hat{\mathcal Z}/\widehat{\mathcal{G}}_1\times\widehat{\mathcal{G}}_2$, suppose it's represented by $(\hat{\mathsf{C}}_1,\hat{\mathsf{C}}_2)\in \hat{\mathcal{Z}}$. Then there exists $g\in \mathcal{G}^{\partial_\infty}$ such that $g\cdot \partial_\infty \hat{\mathsf{C}}_1 =  \partial_\infty \hat{\mathsf{C}}_2$. Suppose $g=\partial_\infty g_1\cdot \partial_\infty g_2$, where $g_i\in \widehat{\mathcal{G}}_i$. Now $([\hat{\mathsf{C}}_1],[\hat{\mathsf{C}}_2])=([g_1\cdot\hat{\mathsf{C}}_1],[g_2^{-1}\cdot\hat{\mathsf{C}}_2])\in \hat{\mathcal Z}/\widehat{\mathcal{G}}_1\times\widehat{\mathcal{G}}_2$ and $(g_1\cdot\hat{\mathsf{C}}_1,g_2^{-1}\cdot\hat{\mathsf{C}}_2)\in \hat{\mathcal{Z}}_\Delta$. So $\hat{\mathcal Z}/\widehat{\mathcal{G}}_1\times\widehat{\mathcal{G}}_2 \subset \hat{\mathfrak{N}}$.
\end{proof}

\begin{corollary}\label{fiberProduct}
\begin{align}
\M(X)&\cong\M(X_0) \times_{\M(\S^1\times \S^2)} \M(\S^1\times D^3)\\
\M(X')&\cong\M(X_0) \times_{\M(\S^1\times \S^2)} \M(D^2\times \S^2)
\end{align}
\end{corollary}
\begin{proof}
By Proposition \ref{prop:triObs} and \ref{prop:trivialObs}, all assumptions of Theorem \ref{thm:ggt} are satisfied. Thus $\M(X)\cong \hat{\mathfrak{N}}$. By Lemma \ref{lem:fibProd},
\[
\M(X)\cong\M(X_0) \times_{\M^{\partial_\infty}(\S^1\times \S^2)} \M(\S^1\times D^3).
\]
But in our case, $H^1(X_0)\to H^1(\S^1\times \S^2)$ is surjective. Thus $\partial_\infty\widehat{\mathcal{G}}_1 =  \mathcal{G}$. Therefore $\M^{\partial_\infty}(\S^1\times \S^2) = \M(\S^1\times \S^2)$. 

The proof of the second equation is similar.
\end{proof}

\subsection{The proof of $1$-surgery formula}\label{subsection:proofUnparemetrized}
Now we can investigate Seiberg-Witten invariants of $X$ and $X'$. According to section 2.2 of \cite{LL01}, for higher dimensional moduli space ${\M}(\hat{N}_r)$, given an integral cohomology class $\Theta$ of moduli space $\hat{\mathcal{B}}_{\hat{N}_r}$, the Seiberg-Witten invariant associate to this class is 
\[
SW^{ \Theta}(\hat{N}_r, \ss) := \langle\Theta,[ {\M}(\hat{N}_r,\ss) ] \rangle
\]

Since $H^1(X)= H^1(X_0) = \R$, $\hat{\mathcal{B}}_X\cong \hat{\mathcal{B}}_{X_0}\cong  \C P^\infty_+\times \S^1$. We choose $\Theta$ to be a generator of $H^1( \C P^\infty_+\times \S^1,\Z)$. 

We first show that the invariant $SW^{ \Theta}$ is well defined:
\begin{lemma}\label{lem:detect-exotic}
Suppose that $b^+(X)>1$ and that $f:X\to X$ is a diffeomorphism. Let $h$ and $k$ be generic paramters. Then $SW^\Theta(E_X, \ss, h) = SW^\Theta(E_X, \ss, k)$.
\end{lemma}
\begin{proof}
Since $b^+(X)>1$, by a generic argement (similar to the one in the proof of \ref{prop:noncompactTransversality-general}), there exists a generic path $K$ from $h$ to $k$. Hence there exists a cobordism from $\M(E_X,\ss,h)$ to $\M(E_X,\ss,k)$. This cobordism is a $2$-dimensional manifold with $1$-dimensional boundary, so after cutting it by the class $\Theta$, we obtain a $1$-dimensional cobordism which gives $SW^\Theta(E_X, \ss, h) = SW^\Theta(E_X, \ss, k)$ (see Figure \ref{fig:infinite-cobordism}).
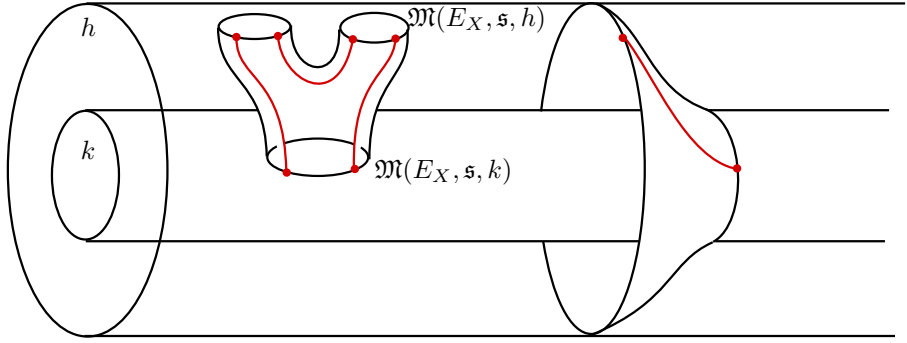
\begin{figure}[ht!]
    \begin{Overpic}{\includegraphics[scale=0.6]{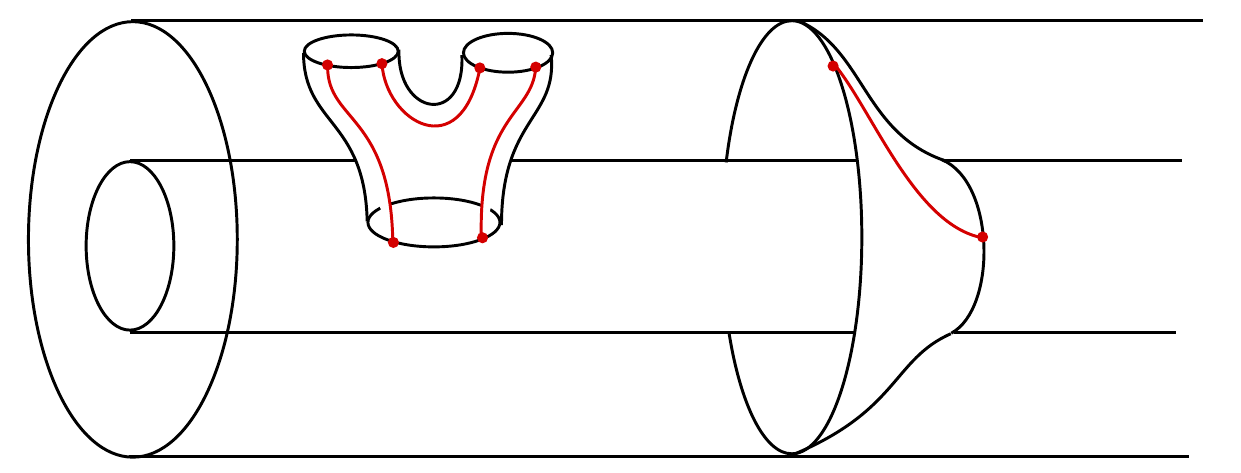}}
     \put(10,33){$h$}
    \put(10,20){$k$}
         \put(44.5,34){$\M(E_X,\ss,h)$}
        \put(41,18){$\M(E_X,\ss,k)$}
        
    \end{Overpic}
    \caption{The cobordism in $\CP^\infty \times \S^1 \times I$.}
    \label{fig:infinite-cobordism}
\end{figure}
\end{proof}

\begin{theorem}\label{main-thm}
$SW^\Theta(X,\ss)=SW(X',\ss')$.
\end{theorem}

\begin{proof}
Since each of $\M(\S^1\times \S^2)$ and $\M(\S^1\times D^3)$ is a circle of reducibles, and these circles are given by the monodromy of connections around their $\S^1$ factor, it's clear that
\[
\partial_\infty: \M(\S^1\times D^3) \to \M(\S^1\times \S^2)
\]
is identity. By Corollary \ref{fiberProduct}, $\M(X)\cong\M(X_0)$.

For $\M(X_0)$, $\partial_\infty:\M(X_0) \to \M(\S^1\times \S^2)$ is not necessarily a homeomorphism, but we can prove that this map is a submersion. Recall that we have choosen a generic perturbation $\eta$ such that $\M(X_0)=\M(X_0,\eta)$ contains only strongly regular points. By the long exact sequence \ref{equ:longExact}:
\[
\cdots
\stackrel{\phi}{\to} H^1_{\hat{\mathsf{C}}_1} = T_{\hat{\mathsf{C}}_1} \M(X_0)
\to H^1(B)=T_{\partial_\infty\hat{\mathsf{C}}_1} \M^{\partial_\infty X_0}(\S^1\times \S^2)
\to H^2(F) =0
\to \cdots
\]
where 
\[
\M^{\partial_\infty X_0}(\S^1\times \S^2) = \mathcal{Z}(\S^1\times \S^2)/\partial_\infty\widehat{\mathcal{G}}_1= \mathcal{Z}(\S^1\times \S^2)/{\mathcal{G}} = \M(\S^1\times \S^2)
\]
(since $H^1(X_0) \stackrel{i^*}{\to} H^1(\S^1\times \S^2) $ is surjective), 
\[
\partial_\infty:\M(X_0) \to \M(\S^1\times \S^2)
\]
is a submersion.

By compactness result, $\M(X_0)$ is a disjoint union of finite many circles, say
$\amalg_{i\in\Gamma} \S^1_i $. Let $d_i$ be the mapping degree of $\partial_\infty|_ {\S^1_i} :  \S^1_i  \to  \M(\S^1\times \S^2) = \S^1$. We claim that 
\[
SW(X,\Theta) = \sum_{i\in \Gamma} d_i.
\]

Let
\[
\hat{\mathfrak{N}}_i :=  \S^1_i \times_{\M(\S^1\times \S^2)} \M(\S^1\times D^3)  \subset \hat{\mathfrak{N}}
\]
be the space of configurations obtained by gluing $\S^1_i$ and $\M(\S^1\times D^3)$. Consider the pullback diagram of moduli spaces:
\begin{equation}
\xymatrix@C=4ex{
\M(X) \subset & \mathcal{B}_X =   \C P^\infty_+\times \S^1 \ar@<-7ex>[d]^{p_1}  \ar[r]^{p_2} & \mathcal{B}_{\S^1\times D^3} = \C P^\infty_+\times \S^1 \ar@<-7ex>[d]^{\partial_\infty^2}  &\supset \M(\S^1\times D^3)=  \{0\} \times\S^1 \ar@<8.3ex>[d]^\cong \\
 \S^1_i \subset & \mathcal{B}_{X_0} =  \C P^\infty_+\times \S^1 \ar[r]^{\partial_\infty^1}   & \mathcal{B}_{\S^1\times \S^2}= \C P^\infty_+\times \S^1  & \supset \M(\S^1\times \S^2) = \{0\} \times\S^1
}
\end{equation}
When restricted to $\S^1$-factors, $\partial_\infty^1$ and $\partial_\infty^2$ are identity maps of $\S^1$, so $p_1$ and $p_2$ are identity maps of $\S^1$. Therefore, $\hat{\mathfrak{N}}_i $ winds around the $\S^1$-factor of $\mathcal{B}_X$ by $d_i$ times. So
\[
\langle [\hat{\mathfrak{N}}_i], \Theta\rangle = d_i.
\]
By Theorem \ref{thm:ggt}, $\M(X)$ is isotopic to $\hat{\mathfrak{N}}$ in $ \mathcal{B}_X $, so 
\[
 \langle[ {\M}(X) ],\Theta \rangle = \sum_{i\in \Gamma} d_i.
 \]

On the other hand, 
\[
\partial_\infty: \M(D^2\times S^2) \to \M(\S^1\times \S^2)
\]
is the inclusion of one point. Thus we have
\begin{equation}
\xymatrix@C=4ex{
\M(X') \subset & \mathcal{B}_{X'}=   \C P^\infty_+\times \S^1 \ar@<-7ex>[d]^{p_1}  \ar[r]^{p_2} & \mathcal{B}_{D^2\times S^2} = \C P^\infty_+ \ar@<-7ex>[d]^{\partial_\infty^2}  &\supset \M(D^2\times S^2)=  \{0\}  \ar@<9ex>[d] \\
 \S^1_i \subset & \mathcal{B}_{X_0} =  \C P^\infty_+\times \S^1 \ar[r]^{\partial_\infty^1}   & \mathcal{B}_{\S^1\times \S^2}= \C P^\infty_+\times \S^1  & \supset \M(\S^1\times \S^2) = \{0\} \times\S^1
}
\end{equation}
Since $\partial_\infty^1|_{ \S^1_i}$ is a submersion, $\hat{\mathfrak{N}}_i :=  \S^1_i \times_{\M(\S^1\times \S^2)} \M(D^2\times S^2)$ contains $d_i$ points. Again by Theorem \ref{thm:ggt}, $\M(X')$ is isotopic to $\hat{\mathfrak{N}}$ in $ \mathcal{B}_{X'} $. So 
\[
SW(X') =  \sum_{i\in \Gamma} d_i = SW(X,\Theta).
\]
\end{proof}

\begin{remark}\label{rem:high-dim-moduli}
Theorem \ref{main-thm} works for $\dim \M(X) >1$ as long as it is odd. In that case we define $SW^\Theta(X,\ss)$ by 
\[
SW^\Theta(X,\ss) := \langle[ {\M}(X) ],\Theta \cup c_1(\C P^\infty)^n \rangle
\]
for $\dim \M(X) =2n+1$. Note that in this case $\dim \M(X') =2n$ and the ordinary invariant is 
\[
SW(X',\ss') := \langle[ {\M}(X') ], c_1(\C P^\infty)^n \rangle.
\]
Hence for $\dim \M(X) >1$, the argument of Theorem \ref{main-thm} follows from a similar proof.
\end{remark}
\section{Applications}

\subsection{Exotic smooth structures on nonsimply connected manifolds}
First observe that by definition and Lemma \ref{lem:detect-exotic}, the cut-down invariant also detects exotic smooth structures. As lots of exotic smooth structures are detected by $SW$, we can now generalize those results to nonsimply connected manifolds by the surgery formula:
\begin{theorem}\label{thm:not-diffeomorphic}
Suppose $X_1$, $X_2$ are two simply connected smooth $4$-manifolds with $b^+_2(X_i)>1$. Suppose $\ss _1$ is a $\text{spin}^c$-structure on $X_1$, such that for any $\text{spin}^c$-structure $\ss_2$ of $X_2$, 
\[
SW(X_1, \ss_1) \neq SW(X_2, \ss_2).
\]
Then $X_1\# (\S^1\times \S^3)$ is not diffeomorphic to $X_2\# (\S^1\times \S^3)$.
\end{theorem}
\begin{proof}
Let $\ss_i'$ be the $\text{spin}^c$-structure of $X_i\# (\S^1\times \S^3)$ such that $\ss_i'$ coincides with $\ss_i$ on the common part. Then by Remark \ref{rem:high-dim-moduli}, 
\[
SW^\Theta(X_1\# (\S^1\times \S^3), \ss_1') \neq SW^\Theta(X_2\# (\S^1\times \S^3), \ss_2').
\]
If there exists a diffeomorphism $f:X_1\to X_2$, by Lemma \ref{lem:detect-exotic}, we have
\[
SW^\Theta(X_1\# (\S^1\times \S^3), \ss_1') = SW^\Theta(X_2\# (\S^1\times \S^3), f(\ss_1')).
\]
Since $H^2(X_2;\Z)\cong H^2(X_2\# (\S^1\times \S^3);\Z)$, there exists a $\text{spin}^c$-structure $\ss_2$ on $X_2$ such that $f(\ss_1') = \ss_2'$. This contradicts the inequality.
\end{proof}

Therefore, we have a lot of exotic nonsimply connected manifolds, for example:
\begin{corollary}
Suppose that $b^+(X) > 1$ and $\pi_1(X) = 1 =\pi_1(X - T )$ where $T$ is a homologically nontrivial torus of self-intersection $0$. Suppose that there exists a $\text{spin}^c$-structure $\ss$ on $X$ such that $SW(X,\ss) \neq 0$. Then $X\# S^1 \times S^3$ admits infinitely many exotic smooth structures. In particular, for the elliptic surface $E(n)$ with $n> 1$, the nonsimply connected manifold $E(n) \# S^1 \times S^3$ admits infinitely many exotic smooth structures.
\end{corollary}
\begin{proof}
For such $X$, Fintushel-Stern knot surgery theorem (see \cite{fintushel1997knotslinks4manifolds}, as well as their lecture notes \cite{fintushel2007lectures4manifolds} Lecture 3) says for any knot $K\subset \S^3$, there exists a manifold $X_K$ homeomorphic to $X$ and 
\[
\max_{n\in \Z}\{SW(X_K,\ss+n[T])\}
\]
depends on the largest coefficient of the Alexander polynomial of $K$. Any symmetric Laurent polynomial whose coefficient sum is $\pm 1$ is the Alexander polynomial of some knot. Hence the set 
\[
\{\max_{n\in \Z}\{SW(X_K,\ss+n[T])\}, K \text{ is a knot in }\S^3\}
\]
is infinite, and therefore we have an infinite family of manifolds that are homeomorphic to $X$ and satisfy the conditions of Theorem \ref{thm:not-diffeomorphic}.
\end{proof}

\subsection{Adjunction formula for odd-dimensional moduli space}
The generalized adjunction formula (\cite{KM94surfaces} and \cite{OS00}) gives a lower bound of the genus of a surface in a smooth $4$-manifold $X$, by Seiberg-Witten basic classes. A characteristic element $K\in H^2(X;\Z)$ is a Seiberg-Witten basic class of $X$ if $SW(X,K)\neq 0$. We can generalize this concept and the adjunction formula to odd dimensional moduli space.

\begin{definition}\label{def:generalized-basic-class}
If $H^1(X;\Z) \cong \Z$, a characteristic element $K\in H^2(X;\Z)$ is a Seiberg-Witten basic class of $X$ if:
\begin{enumerate}
\item[\textbullet] $SW(X,K)\neq 0$ for $\dim \M(X,K) =2n$;
\item[\textbullet]  $SW^\Theta(X,K)\neq 0$ for $\dim \M(X,K) =2n+1$.
\end{enumerate}
If $H^1(X;\Z) \neq \Z$, define the Seiberg-Witten basic class of $X$ as usual.
\end{definition}

We can also generalize the concept of the simple type:
\begin{definition}
A simply connected $4$-manifold $X$ is of simple type if each basic class $K$ satisfies $\dim \M(X,K) =0$. A $4$-manifold $X$ with $H^1(X;\Z) \cong \Z$ is of simple type if each basic class $K$ satisfies $\dim \M(X,K) = 1$.
\end{definition}

The following theorem has the same form as the generalized adjunction formula, but with our generalization of the basic class, the following formula will give more infomation for the nonsimply connected manifolds:
\begin{theorem}[Generalized adjunction formula]
Suppose that $\Sigma$ is an embedded, oriented, connected, homologically nontrivial surface in $X$ with genus $g(\Sigma)$ and self-intersection $[\Sigma]^2 \ge 0$. Then for every Seiberg-Witten basic class $K\in H^2(X;\Z)$, we have
\[
2g(\Sigma) -2 \ge [\Sigma]^2 + |K([\Sigma])|.
\]
If $X$ is of simple type and $g(\Sigma)>0$, then the same inequality holds without requiring $[\Sigma]^2 \ge 0$.
\end{theorem}
\begin{proof}
Let $K\in H^2(X;\Z)$ be a basic class. If $\dim \M(X,K)$ is even, the theorem follows from the adjunction formula for ordinary Seiberg-Witten invariant. 

Now suppose $\dim \M(X,K)$ is odd. If $H^1(X;\Z) \neq \Z$, then by the definition of the ordinary Seiberg-Witten invariant, $SW(X,K)= 0$. So $K$ is not a basic class. If $H^1(X;\Z) \cong \Z$, then by Definition \ref{def:generalized-basic-class} we have $SW^\Theta(X,K)\neq 0$. Let $\gamma$ be a loop in $X$ which is sent to $1$ by a generator of $H^1(X;\Z)$. Denote the resulting manifold after a surgery along $\gamma$ by $X'$. 

Now since $\gamma$ is $1$-dimensional, and $\Sigma$ is $2$-dimensional, we can assume that $\Sigma$ and a neighborhood $\S^1\times D^3$ of $\gamma$ are disjoint in $X$, so the surgery doesn't change this surface (just as what we did in the proof of Theorem \ref{thm:changeOfSpinc}). Denote the resulting surface by $\Sigma' \subset X'$. The surgery would not change the self-intersection of this surface (see the last part of the proof of Proposition \ref{prop:virtualDimX0}), so $[\Sigma]^2 =[\Sigma']^2$. By two Mayer-Vietoris sequences of cohomology groups (they are in the proof of Theorem \ref{thm:changeOfSpinc} and Proposition \ref{prop:virtualDimX0}), there exists an isomorphism $H^2(X;\Z) \to H^2(X';\Z) $. Let $K'$ be the image of $K$ under this isomorphism. Then by Remark \ref{rem:high-dim-moduli} we have 
\[
SW^\Theta(X,K)=SW(X',K') \neq 0.
\]
Hence $K'$ is a basic class on $X'$. Note that the dual of $K$ and the dual of $K'$ are the same surface in $X_0 := X - \S^1\times D^3 = X' - D^2 \times \S^2$. Also $\Sigma$ and $\Sigma'$ are the same surface in $X_0$. Hence 
\[
|K'([\Sigma'])| = |K'([\Sigma'])|.
\]
It's easy to check that $\Sigma'$ satisfies all requirements of the adjunction formula for ordinary Seiberg-Witten invariant. Hence
\[
2g(\Sigma') -2 \ge [\Sigma']^2 + |K'([\Sigma'])|.
\]
Therefore we have
\[
2g(\Sigma) -2 \ge [\Sigma]^2 + |K([\Sigma])|.
\]

Lastly we prove that if $X$ is of simple type and $g(\Sigma)>0$, then the same inequality holds without requiring $[\Sigma]^2 \ge 0$. When $X$ is simply connected, this follows from the adjunction formula for ordinary Seiberg-Witten invariant. When $H^1(X;\Z) \cong \Z$, by Theorem \ref{main-thm} and Remark \ref{rem:high-dim-moduli}, $X'$ is of simple type. Hence the statement follows from the adjunction formula for the simply connected $4$-manifold $X'$.
\end{proof}

\bibliographystyle{alpha}
\bibliography{./diff}
\end{document}

%% file: header.tex
\usepackage[pagebackref]{hyperref}
\usepackage{pinlabel}
\usepackage{overpic}
\usepackage{svg}
\usepackage{sseq}
\hypersetup{
    colorlinks = true,
    citecolor = [rgb]{0, 0.5, 0.00},
}

\renewcommand*{\backref}[1]{}
\renewcommand*{\backrefalt}[4]{%
    \ifcase #1 (Not cited.)%
    \or        (Cited on page~#2.)%
    \else      (Cited on pages~#2.)%
    \fi}

\usepackage{mathtools}
\usepackage{amssymb}

\usepackage[margin=1.175in]{geometry}
\usepackage{todonotes}
\usepackage{multirow}
\usepackage{longtable}
\usepackage{enumerate}
\usepackage{mathabx}
\usepackage{amsmath}
\usepackage{amsfonts}
\usepackage{amssymb}
\usepackage{float}
\usepackage{amsthm}
\usepackage{comment}
\usepackage{amscd}
\usepackage{amsxtra}
\usepackage{epsfig}
\usepackage{slashed}

\usepackage{listings}
\usepackage{color}

\definecolor{dkgreen}{rgb}{0,0.6,0}
\definecolor{gray}{rgb}{0.5,0.5,0.5}
\definecolor{mauve}{rgb}{0.58,0,0.82}

\usepackage{color}
\usepackage[all]{xy}
\usepackage{verbatim}

\usepackage{eucal}
\usepackage{mathrsfs}

\usepackage{graphicx}
\usepackage{tikz-cd}

\usepackage{url}
\usepackage{verbatim}

\usepackage{xy}
\xyoption{all}

\usepackage{amsthm}
\usepackage{tikz}
\usetikzlibrary{decorations.pathreplacing}

\usepackage{caption} 
\makeatletter
\def\@tocline#1#2#3#4#5#6#7{\relax
  \ifnum #1>\c@tocdepth 
  \else
    \par \addpenalty\@secpenalty\addvspace{#2}%
    \begingroup \hyphenpenalty\@M
    \@ifempty{#4}{%
      \@tempdima\csname r@tocindent\number#1\endcsname\relax
    }{%
      \@tempdima#4\relax
    }%
    \parindent\z@ \leftskip#3\relax \advance\leftskip\@tempdima\relax
    \rightskip\@pnumwidth plus4em \parfillskip-\@pnumwidth
    #5\leavevmode\hskip-\@tempdima
      \ifcase #1
       \or\or \hskip 1em \or \hskip 2em \else \hskip 3em \fi%
      #6\nobreak\relax
    \hfill\hbox to\@pnumwidth{\@tocpagenum{#7}}\par
    \nobreak
    \endgroup
  \fi}
\makeatother

\DeclareMathOperator{\im}{im}
\DeclareMathOperator{\imaginary}{\mathfrak{Im}}

\newtheorem{lemma}{Lemma}[section]

\newtheorem{corollary}[lemma]{Corollary}

\newtheorem{theorem}[lemma]{Theorem}

\newtheorem{prop}[lemma]{Proposition}

\theoremstyle{definition}

\newtheorem{definition}[lemma]{Definition}

\newtheorem{example}[lemma]{Example}

\newtheorem{remark}[lemma]{Remark}

\numberwithin{equation}{section} 
\newcommand{\Z}{\mathbb{Z}}

\newcommand{\M}{\mathfrak{M}}

\renewcommand{\L}{\mathcal{L}}

\newcommand{\CP}{\mathbf{C}P}
\newcommand{\RP}{\mathbf{R}P}

\renewcommand{\L}{\mathcal{L}}

\newcommand{\C}{\mathbb{C}}

\newcommand{\e}{\mathrm{e}}

\newcommand{\R}{\mathbb{R}}

\newcommand{\D}{\slashed{\mathfrak D}}

\renewcommand{\S}{\mathbb{S}}
\renewcommand{\ss}{\mathfrak{s}}

\newcommand{\SO}{\mathrm{SO}}

\newcommand{\SU}{\mathrm{SU}}

\renewcommand{\H}{\mathcal{H}}

